\documentclass[12pt,letterpaper]{article}
\usepackage{graphics,epsfig,graphicx,color}
\graphicspath{{/EPSF/}{../Figures/}{Figures/}}

\usepackage{amssymb,latexsym}

\usepackage{setspace}
\usepackage{amsmath}

\usepackage{mathrsfs,amsfonts}

\usepackage{amsthm,amsxtra}

\newcommand\eq[1] {(\ref{#1})}

\newtheorem{lemma}{Lemma}[section]

\newtheorem{corollary}{Corollary}[section]
\newtheorem{question}{Question}
\newcommand{\bfm}[1]{\mbox{\boldmath ${#1}$}}

\newcommand{\beqa}{\begin{eqnarray}}
\newcommand{\eeqa}[1]{\label{#1}\end{eqnarray}}
\newcommand{\beq}{\begin{equation}}
\newcommand{\eeq}[1]{\label{#1}\end{equation}}
\newcommand{\Gd}{\delta}

\newcommand{\GD}{\Delta}

\newcommand{\BGv}{\bfm\nu}

\newcommand{\CD}{{\cal D}}

\def\Bx{{\bf x}}
\def\By{{\bf y}}
\def\Bz{{\bf z}}

\def\B0{{\bf 0}}

\def \ph {\varphi}
\def \RR {{\mathbb R}}

\def \ba {\begin{array}}
\def \ea {\end{array}}
\newtheorem {Thm} {Theorem} [section]

\newtheorem {Pro} [Thm] {Proposition}
\newtheorem {Adef} [Thm] {Definition}

\newtheorem {Arem} [Thm] {Remark}

\newtheorem {Aexa} [Thm] {Example}

\newtheorem {Anot} [Thm] {Notation}

\def \refe #1.{(\ref{#1})}
\def \reff #1.{figure~\ref{#1}}
\def \refs #1.{section~\ref{#1}}
\def \refss #1.{subsection~\ref{#1}}
\def \refD #1.{Definition~\ref{#1}}
\def \refT #1.{Theorem~\ref{#1}}
\def \refL #1.{Lemma~\ref{#1}}
\def \refC #1.{Corollary~\ref{#1}}
\def \refP #1.{Proposition~\ref{#1}}
\def \refR #1.{Remark~\ref{#1}}
\def \refE #1.{Example~\ref{#1}}
\def \refN #1.{Notation~\ref{#1}}

\newcommand{\eps}{\varepsilon}

\newcommand{\lnn}{\mbox{ln}}
\newcommand{\Ker}{{\rm Ker}}
\newcommand{\dist}{\mbox{dist}}

\newif\ifPDF
\ifx\pdfoutput\undefined
    \PDFfalse
\else
    \ifnum\pdfoutput > 0
        \PDFtrue
    \else
        \PDFfalse
    \fi
\fi

\ifPDF
    \usepackage{pdftricks}
    \begin{psinputs}
        \usepackage{pstricks}
        \usepackage{pstcol}
        \usepackage{pst-plot}
        \usepackage{pst-tree}
        \usepackage{pst-eps}
        \usepackage{multido}
        \usepackage{pst-node}
        \usepackage{pst-eps}
    \end{psinputs}
\else
    \usepackage{pstricks}
\fi

\ifPDF
    \usepackage[debug,pdftex,colorlinks=true,
    linkcolor=blue, bookmarksopen=false,
    plainpages=false,pdfpagelabels]{hyperref}
\else
    \usepackage[dvips]{hyperref}
\fi

\usepackage{fancyhdr}

\newenvironment{keywords}
{\noindent{\bf Key words.}\small}{\par\vspace{1ex}}

\title{On the active manipulation of quasistatic fields and its applications}

\author{
    Daniel Onofrei\thanks{
        Department of Mathematics,
        University of Houston,
        Houston, Texas 77004; onofrei@math.uh.edu
    }
}

\begin{document}
%%%%%%%%%%%%%%%%%%%%%%%%%%%%%%%%%%%%%%%%%%%%%%%%%%%%%%%%%%%%%%%%%%
%%%%%%%BEGIN DOCUMENT %%%%BEGIN DOCUMENT %%%%BEGIN DOCUMENT%%%%%%%
%%%%%%%%%%%%%%%%%%%%%%%%%%%%%%%%%%%%%%%%%%%%%%%%%%%%%%%%%%%%%%%%%%

\maketitle

\tableofcontents

%%%%%%%%%%%%%%%%%%%%%%%%%%%
%%%%%%%%ABSTRACT%%%%%%%%%%%
%%%%%%%%%%%%%%%%%%%%%%%%%%%

\begin{abstract}
Following the ideas proposed in \cite{OMV1} and \cite{OMV2} on active exterior cloaking, we present here a systematic integral equation method to generate suitable quasistatic fields for cloaking, illusions and energy focusing (with given accuracy) in multiple regions of interests. In the quasistatic regime, the central issue is to design appropriate source functions for the Laplace equation so that the resulting solution will satisfy the required properties. We show the existence and non-uniqueness of solutions to the problem and study the physically relevant unique $L^2$-minimal energy solution. we also provide some numerical evidences on the feasibility of the proposed approach.
\end{abstract}

%%%%%%%%%%%%%%%%%%%%%%%%%%%
%%%%%%%%%KEYWORDS%%%%%%%%%%
%%%%%%%%%%%%%%%%%%%%%%%%%%%


%%%%%%%%%%%%%%%%%%%%%%%%%%
%%%   AMS or PACS   %%%%%%
%%%%%%%%%%%%%%%%%%%%%%%%%%

%\begin{AMS}
%35Q55, 60H15.
%\end{AMS}

%%%%%%%%%%%%%%%%%%%%%%%%%%%%%%%%%%%%%%%%%%%%%%%%%%%%%%%%%%%%%%%%%%
%%%%%%%%%%%%%%%%%%%%%%%%%%%%%%%%%%%%%%%%%%%%%%%%%%%%%%%%%%%%%%%%%%
\section{Introduction}
%%%%%%%%%%%%%%%%%%%%%%%%%%%%%%%%%%%%%%%%%%%%%%%%%%%%%%%%%%%%%%%%%%
%%%%%%%%%%%%%%%%%%%%%%%%%%%%%%%%%%%%%%%%%%%%%%%%%%%%%%%%%%%%%%%%%%

The technique of manipulating acoustic and electromagnetic fields
in desired regions of space has been greatly advanced in the recent
years, mainly due to its fascinating applications such as,
cloaking, the creation of illusions, secret remote communication,
focusing energy, and novel imaging techniques. The development can
be roughly classified into two categories. 

The first type of techniques attempts to passively control the fields in the regions of interest by
changing the material properties of the medium in certain surrounding regions
while the second type of schemes  focus on the active manipulation (active control) of fields with the help of specially designed sources.

In \cite{Green1} the authors presented the first rigorous discussion of the passive manipulation of fields in the context of
quasistatics cloaking (see also 
\cite{Dol}, \cite{Post} and \cite{LaxN} where the invariance to a change of variables is fully explained and the transformed material are fully described), and was later extended in \cite{Pendry} to the
general case of passive manipulation of fields in the finite frequency regime (see also the review \cite{Chan2}
and references therein). These passive strategies are now known as ``transformation optics''.
The similar strategy in context of
acoustics was proposed in \cite{Cummer} (see also the review
\cite{Chan3} and references therein). The idea behind
transformation optics/acoustics is the invariance of the
corresponding Dirichlet to Neumann-map (boundary measurements map)
considered on some external boundary with respect to suitable
change of variables which are identity on the respective boundary.
This result implies that two different materials (the initial one
and the one obtained after the change of variables is applied)
occupying some region of space $\Omega$, will have the same
boundary measurements maps on $\partial \Omega$ and thus be
equivalent from the point of view of an external observer. This
leads to a long list of important applications, such as, cloaking,
field concentrators (\cite{Gunther1}) or field rotators, illusion optics, etc. (see
\cite{Chan2}, \cite{Chan3}, \cite{Green2}, \cite{Alu} and references therein), cloaking sensors while maintaining their sensing capability \cite{Gunther2}, \cite{ALU4},

Recently, in an effort to improve accuracy and stability of the transformation optics/acoustics, various regularization of this scheme have been studied (see \cite{Kohn} and references therein, \cite{Hongyu3}, \cite{Hoai}, \cite{Hongyu2}, \cite{Kang3},  \cite{Kang1},  \cite{Kang2}). Positive results about generating broadband low-loss metamaterial response have been obtained in \cite{Shalaev}, \cite{Valentine} and a new, more stable regularization strategy was recently proposed in \cite{Hongyu1}. 

In a parallel direction, many researchers focused on  other alternative field manipulation strategies. They can be grouped into two main categories, passive designs based either on artificial materials with extreme properties or on geometrical arguments, and active designs based on the active control of fields by only using antennas with no materials needed in the scheme, 

Among the alternative passive techniques proposed in the literature we could mention, plasmonic
designs (see \cite{Alu} and references therein), strategies based
on anomalous resonance phenomena (see \cite{Mil1}, \cite{Mil2},
\cite{Mil3}), conformal mapping techniques (see
\cite{Ulf1},\cite{Ulf2}), and complimentary media strategies (see
\cite{Chan}).

	Regarding the active designs for the manipulation of fields we mention that this idea appeared first in the 
 	context of low-frequency acoustics where various techniques for the active control of low-frequency sound 
 	(or active noise cancellation) were proposed in the literature, and we could mention here the pioneer works of Leug \cite{Leug} (feedforward control of sound) and Olson \& May \cite{Olson-May} (feedback control of sound). For a more detailed account of very interesting recent developments of the idea in the context of acoustics we mention the reviews \cite{Tsynkov}, \cite{Peterson}, \cite{Fuller}, \cite{Peake} \cite{Elliot} and the references therein. 
 	
In the electromagnetic regime, several active designs have been recently proposed in the literature and we could mention
the interior active cloaking strategy proposed in \cite{Miller}
which uses active boundaries and the exterior active cloaking
scheme discussed in \cite{OMV1}, \cite{OMV2}, \cite{OMV3},
\cite{OMV4} (see also \cite{CTchan-num}) which uses a discrete number of active sources
(antennas) to manipulate the fields. The active exterior strategy
for 2D quasistatics cloaking was introduced in \cite{OMV1}, were
based on \emph{a priori} information about the incoming field,
with the help of one active source (antenna), we constructively
described how one can create an almost zero field external region
while maintaining a very small scattering effect in the far field.
The proposed strategy did not work for objects closed to the
antennas, it cloaked large objects only when they are far enough
to the antenna (see \cite{OMV4}) and was not adaptable for three
space dimensions. The finite frequency case was studied in the
last section of \cite{OMV1} and in \cite{OMV3} (see also
\cite{OMV4} for a recent review) where three active sources
(antennas) were needed to create a zero field region in the
interior of their convex hull while creating a very small
scattering effect in the far field. The broadband character of the
proposed scheme was numerically observed in \cite{OMV2}. We
mention now that, from the point of view of the possible
applications the constraint that the antennas surround the region
of interests is not desirable and one would like to find a
solution for the active manipulation of fields by using only one
active source (antenna) as we proposed in \cite{OMV1}. 

In this paper, we address the problem formulated in Question \ref{QUST:Main} for the
particular case of the quasistatic regime and a homogeneous
environment). This problem is of course ill-posed and this explains the
multitude of possible approximate solutions proposed for it. Our
aim is to provide a unified mathematical theory for the general
problem of active manipulation of electromagnetic or acoustic
fields, which will work in a broadband regime and regardless of
dimension, will allow for robust computational simulations and for
the approximation of a stable optimal energy solution and will be
appropriate for the more general case of non-homogeneous
environments. In the present work we introduce the mathematical
theory and analyse the problem in the quasistatic regime (modelled by the Laplace operator)
corresponding to a homogeneous environment.

The paper is organized as follows. In Section~\ref{SEC:Form} we
formulate mathematically the problem of generating desired field
in certain regions of space using active sources. We then study in
Section~\ref{SEC:Exist} the existence of solutions of the
mathematical problem and Section~\ref{SEC:MES} the the
constructive approximation of a solution with minimum energy. We
provide some numerical simulations to support our theoretical
results in Section~\ref{SEC:Num}. Concluding and further remarks
are offered in Section~\ref{SEC:Concl}. Finally, for the sake of
completeness, we added the proofs for two technical results in the
Appendix.

%%%%%%%%%%%%%%%%%%%%%%%%%%%%%%%%%%%%%%%%%%%%%%%%%%%%%%%%%%%%%%%%%%
%%%%%%%%%%%%%%%%%%%%%%%%%%%%%%%%%%%%%%%%%%%%%%%%%%%%%%%%%%%%%%%%%%
\section{Problem formulation}
\label{SEC:Form}
%%%%%%%%%%%%%%%%%%%%%%%%%%%%%%%%%%%%%%%%%%%%%%%%%%%%%%%%%%%%%%%%%%
%%%%%%%%%%%%%%%%%%%%%%%%%%%%%%%%%%%%%%%%%%%%%%%%%%%%%%%%%%%%%%%%%%

Let $D_\delta\subset\RR^d$ ($d=2,3$) be a small neighborhood of
the origin $\B0$, and $D$ a given smooth domain containing
$D_\delta$. Let the regions of interest, $\{D_k\}_{k=1}^N$, be $N$
subdomains of $D$ (i.e. $D_k\subset\subset D$, $1\le k\le N$) that
are disjoint in the sense that
$\overline{D}_k\cap\overline{D}_{k'}=\emptyset, \forall k\neq k'$.
We also require that $D_\delta$ be disjoint with $D_k$,
$\overline{D}_\delta\cap \overline{D}_k=\emptyset, \mbox{ for all
} k$. We denote $u_0$ a smooth function on $\RR^d\backslash
\overline{D}$, and by $u_k$ a smooth function that is harmonic in
a neighborhood of $D_k$, i.e. $\Delta u_k=0$ in $V\subset \RR^d$
with $D_k\subset\subset V$. Then the general mathematical question
that we want to ask in quasistatic regime is:

\begin{question}\label{QUST:Main}
Can we design an exterior active source (antenna), modeled as a
continuous function $h(\Bx)$ supported on $\partial D_\delta$,
such that the harmonic field in $\RR^d\backslash D_\delta$
generated by $h(\Bx)$, say $u$, has the property that $u\approx
u_0$ in $\RR\backslash \overline{D}$ and $u\approx u_k$ in $D_k$
for all $1\le k\le N$, where by $\approx$ we mean a good
approximation in the uniform convergence norm?
\end{question}

This question appears naturally in many applications.
For instance, if the answer to Question~\ref{QUST:Main} is positive,
then one can use the active source (antenna) on $D_\delta$ to generate
a zero field in $\bigcup_{k=1}^N D_k$ and a scattering field $u_0$ corresponding
to an arbitrary object in $\RR^d\backslash\overline{D}$ to create an illusion
for an external (outside of $\overline{D}$) observer. One can also program the
active source (antenna) to approximate $N$ different desired fields in each of
the regions $D_k$, $1\le k\le N$ while creating a zero field region in $\RR^d\backslash\overline{D}$,
thus sending information to regions of interests without being detected by an outside observer.

%Then, from the physical point of view, we say that a certain region $D\subset\RR^d$, compact subset of a larger region $\Omega\subset\RR^d$, is actively cloaked from an incoming probing field $\Bu_i$, if there exists a field $\Bu_d$ to be generated by an external cloaking device such that $\Bu_d+\Bu_i\approx 0$ in $D$ and $\Bu_d\approx 0$ outside the larger region $\Omega$.

We now study Question~\ref{QUST:Main} in more detail. To simplify
the presentation, but without loss of generality,  we assume that
all regions involved in Question~\ref{QUST:Main} are balls in
$\RR^d$. We denote by $B_r(\Bx)$ the d-dimensional open ball that
centered at $\Bx\in\RR^d$ with radius $r>0$. Moreover, we first
present the case where only one region of interests is involved
and then, in Remark \ref{Rem1}, show how the general result (i.e.,
the case of $N$ region of interests) follows as an immediate
consequence. Thus, let us consider Question~\ref{QUST:Main} with
$N=1$, $D_\Gd=B_\Gd(0)$, $D=B_R(\B0)$ and $D_1=B_a(\Bx_0)$ and
(for technical reasons to be discussed later) new parameters $R'$,
and $a'$ such that
\begin{equation}
\label{par-cond}
    a<a',\ R'<R,\ |\Bx_0|>a'+\delta,\ \mbox{ and }\ R'>|\Bx_0|+a'.
\end{equation}
A schematic illustration of the problem setting and various
geometrical parameters are show in Fig.~\ref{FIG:Geo}. Then, in
the case when $u_0$ denotes a homogeneous quasistatic potential
Question~\ref{QUST:Main} can be formulated mathematically as
follows.

\paragraph{\bf Formulation A.}
 Let $0<\eps\ll 1$ be fixed. Find a function $h\in C(\partial B_\delta(\B0) )$ such that there exists $v\in C^2(\RR^d\setminus
 \overline{B}_\delta(\B0))\cap C^1(\RR^d\setminus
 B_\delta(\B0))$ solution of,
 \beq\vspace{0.15cm}\left\{\vspace{0.15cm}\begin{array}{llll}
 \GD v=0 \mbox{ in }\RR^d\setminus \overline{B}_\Gd(\B0)\vspace{0.15cm}\\
 v=h \mbox{ on }\partial B_\Gd(\B0)\vspace{0.15cm}\\
 \Vert v-u_1\Vert_{C(\bar{B}_a(\Bx_0))}\le \eps\vspace{0.15cm}\\
 \Vert v-u_0\Vert_{C(\RR^d\setminus B_R(\B0))}\le \eps\end{array}\right.
 \eeq{2} where $u_1$ is a given function harmonic in a set containing $B_a(\Bx_0)$ and the norm $\|\cdot\|_{C(X)}$ is the
usual uniform norm on continuous functions defined on $X$.
\begin{figure}[ht]
    \centering
    \includegraphics[trim=0cm 0cm 0cm 0cm,clip,angle=0,width=0.4\textwidth]{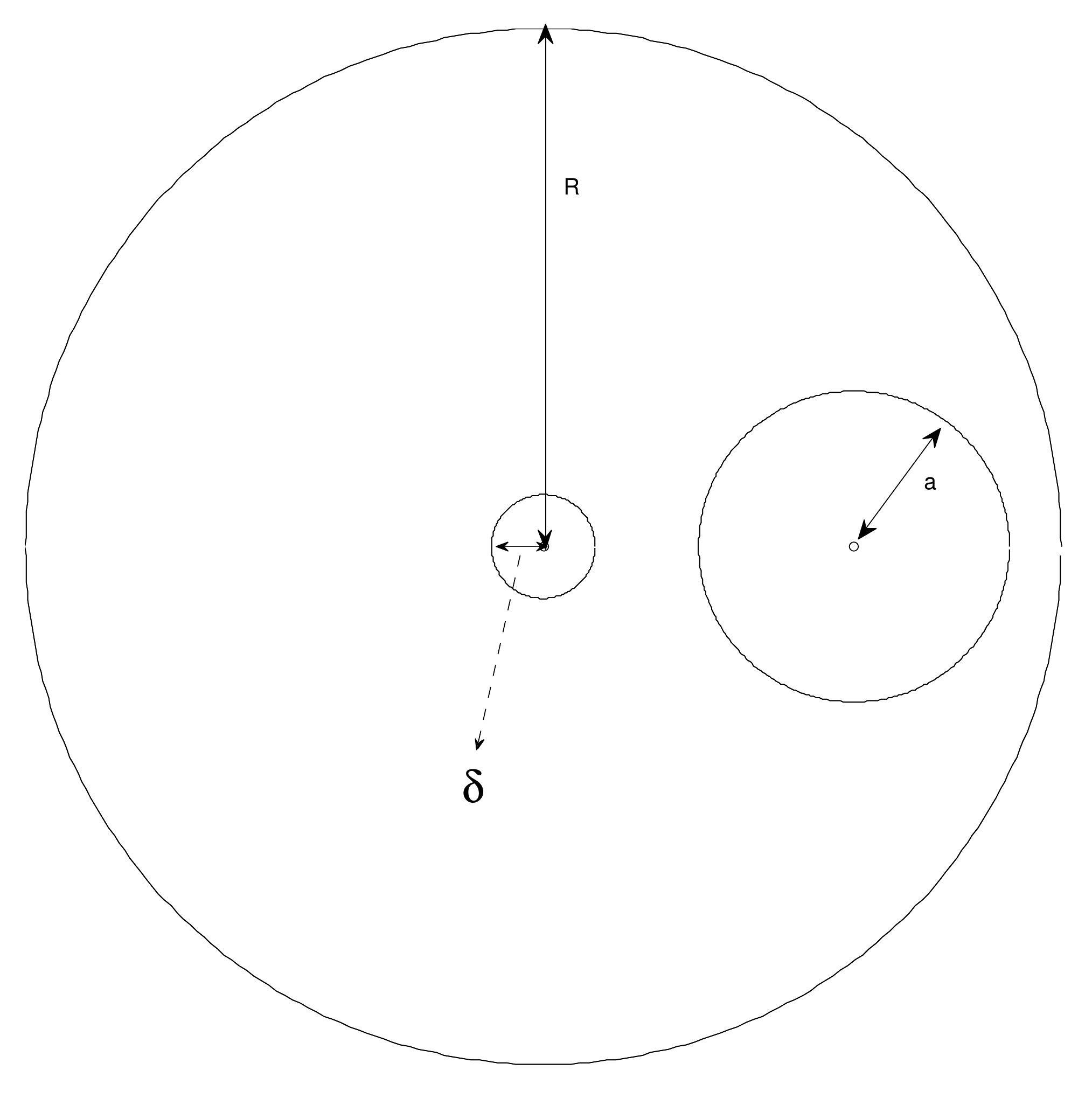}
    \caption{The geometrical setting of Formulation A.}
    \label{FIG:Geo}
\end{figure}

If we subtract $u_0$ from $v$ in Formulation A, and denote by
$u\equiv v-u_0$, $g\equiv h-u_0$, we obtain an equivalent
formulation of the original problem.
\paragraph{\bf Formulation A$'$.}
Let $0<\eps\ll 1$ be fixed. Find a function $g\in C(\partial B_\delta(\B0) )$ such that there exists $u\in C^2(\RR^d\setminus\overline{B}_\delta(\B0))\cap C^1(\RR^d\setminus
 B_\delta(\B0))$ solution of,
\beq\vspace{0.15cm}\left\{\vspace{0.15cm}\begin{array}{llll}
 \GD u=0 \mbox{ in }\RR^d\setminus \overline{B}_\Gd(\B0)\vspace{0.15cm}\\
 u=g \mbox{ on }\partial B_\Gd(\B0)\vspace{0.15cm}\\
 \Vert u+u_0-u_1\Vert_{C(\bar{B}_a(\Bx_0))}\le \eps\vspace{0.15cm}\\
 \Vert u\Vert_{C(\RR^d\setminus B_R(\B0))}\le \eps\end{array}\right.
 \eeq{3}
Thus a solution of our problem is a function $g\in C(\partial
B_\delta)$ (resp. $h\in C(\partial B_\delta)$ for \eq{2}), such
that there exists at least a solution for problem \eq{3} (resp.
\eq{2}). Such a solution will describe the required potential to
be generated at the active source (antenna) so that an
approximation of $0$ (resp. $-u_0$) in the region $B_a(\Bx_0)$
with $\eps$-accuracy will be possible with a very small
perturbation of the far field (resp. very small far field).

%Further in the paper, without loosing the generality we will assume  that $\Bx_0=(c,0)$ if $d=2$ or $\Bx_0=(c,0,0)$ if $d=3$ and for any $r>0$ we will write $B_r(c)$ for $B_r(\Bx_0)$.

Let $a', R', \Bx_0, \Gd$ be as before. We introduce the following space $\Xi$,
\begin{equation}\label{EQ:Parameters}
    \Xi \equiv L^2(\partial B_{a'}(\Bx_0))\times L^2(\partial B_{R'(\B0)})
\end{equation}
Then $\Xi$ is a Hilbert space with
respect to the scalar product given by
\begin{equation}\label{EQ:IP}
(\varphi,\psi)_{\Xi}=\int_{\partial
B_{a'}(\Bx_0)}\varphi_1(\By)\psi_1(\By)  ds_{\By}+\int_{\partial
B_{R'}(\B0)}\varphi_2(\By)\psi_2(\By) ds_{\By}
\end{equation}
for all $\varphi \equiv (\varphi_1,\varphi_2)$ and $\psi=(\psi_1,\psi_2)$ in
$\Xi$. The next lemma presents two technical regularity results
which, in order to make the paper self contained, will be proved
in the Appendix.

\begin{lemma}\label{lemma-1}
Let $0<R_1<R_*<R_2$ be three constants and $\By_0\in\RR^d$ an arbitrary point.
%Consider three concentric balls centered in some arbitrary point, ${\By}_0\in\RR^d$, i.e., $B_{R_1}=
%B_{R_1}({\By}_0)$, $B_{R_*}=B_{R_*}({\By}_0)$, and $B_{R_2}=B_{R_2}({\Bx}_0)$, with $R_1<R_*<R_2$.
Let $f,g\in C(\partial B_{R_*}(\By_0))$ and define $v_i\in C^2(B_{R_*}(\By_0))\cap C^1(\overline{B}_{R_*}(\By_0))$ and
$v_e\in C^2(\RR^d\setminus \overline{B}_{R_*}(\By_0))\cap C^1(\RR^d\setminus B_{R_*}(\By_0))$ to be the solutions of the following interior and exterior Dirichlet problems respectively,
\beq
\left\{\begin{array}{ll}
\GD v_i =0 \mbox{ in } B_{R_*}(\By_0)\vspace{0.15cm}\\
v_i=f \mbox{ on } \partial
B_{R_*}(\By_0)\vspace{0.15cm}\end{array}\right.
\eeq{6}
and
\beq
\left\{\begin{array}{lll}\GD v_e =0 \mbox{ in }\RR^d\setminus \bar
B_{R_*}(\By_0)\vspace{0.15cm}\\
v_e=g \mbox{ on } \partial
B_{R_*}(\By_0)\vspace{0.15cm} \\
v_e=\left\{\begin{array}{ll}O(1)\mbox{ for } |\Bx|\rightarrow
\infty,
\mbox{ if } d=2\vspace{0.15cm}\\
o(1)\mbox{ for } |\Bx|\rightarrow \infty, \mbox{ if }
d=3\end{array}\right.\end{array}\right.
\eeq{7}
Then we have,
$$(i)\; \displaystyle\Vert v_i\Vert_{C(\bar{B}_{R_1}(\By_0))}\leq
\frac{R_*+R_1}{|B_1|R_*(R_*-R_1)^{d-1}}\Vert f\Vert_{L^1(\partial
B_{R_*}(\By_0))} $$
$$(ii)\; \displaystyle\Vert v_e\Vert_{C(\RR^d\setminus
B_{R_2}(\By_0))}\leq \frac{R_2+R_*}{|B_1|R_*(R_2-R_*)^{d-1}}\Vert
g\Vert_{L^1(\partial B_{R_*}(\By_0))}$$
where $|B_1|$ denotes the volume of the unit ball $B_1(\By_0)$.
\end{lemma}
The Big $O$ and little $o$ notations in the radiation condition
guaranteeing the uniqueness of the solution for the exterior
problem are the standard ones.

%%%%%%%%%%%%%%%%%%%%%%%%%%%%%%%%%%%%%%%%%%%%%%%%%%%%%%%%%%%%%%%%%%
%%%%%%%%%%%%%%%%%%%%%%%%%%%%%%%%%%%%%%%%%%%%%%%%%%%%%%%%%%%%%%%%%%
\section{Existence of solutions}
\label{SEC:Exist}
%%%%%%%%%%%%%%%%%%%%%%%%%%%%%%%%%%%%%%%%%%%%%%%%%%%%%%%%%%%%%%%%%%
%%%%%%%%%%%%%%%%%%%%%%%%%%%%%%%%%%%%%%%%%%%%%%%%%%%%%%%%%%%%%%%%%%

We are now ready to present the main results. Let us introduce the integral operator,
$K:L^2(\partial B_\Gd(\B0))\rightarrow \Xi$, defined as
\begin{equation}\label{EQ:K}
    Ku(\Bx,\Bz) =  (K_1u(\Bx), K_2u(\Bz))
\end{equation}
for any $u(\Bx)\in L^2(\partial B_\delta(\B0))$, where
\beqa
K_1u(\Bx)&=&\int_{\partial B_\Gd(\B0)}u(\By)\frac{\partial
\Phi(\Bx,\By)}{\partial \BGv_{\By}}ds_{\By}, \mbox{ for }\Bx\in
\partial B_{a'}(\Bx_0)\nonumber\\
&&\nonumber\\
K_2u(\Bz)&=&\int_{\partial B_\Gd(\B0)}u(\By)\frac{\partial
\Phi(\Bz,\By)}{\partial \BGv_{\By}}ds_{\By}, \mbox{ for }\Bz\in
\partial B_{R'}(\B0)
\eeqa{8}
where $\displaystyle\BGv_{\By}=\frac{\By}{|\By|}$ is the normal exterior to $\partial B_\Gd(\B0)$ and where $\Phi(\Bx, \By)$
represents the fundamental solution of the Laplace operator, i.e.,
\beq
\Phi(\Bx,\By)=\left\{\begin{array}{ll}\vspace{0.15cm}\displaystyle\frac{1}{2\pi}\lnn
\frac{1}{|\Bx-\By|},\mbox{
for } d=2\vspace{0.15cm}\\
\displaystyle\frac{1}{4\pi}\frac{1}{|\Bx-\By|},\mbox{ for }
d=3\end{array}\right.
\eeq{9}
The next result is classical but, for
the sake of completeness, we included its proof in the Appendix.

\begin{lemma}\label{lemma-2}
    The operator $K$ defined in~\eq{EQ:K} is a compact linear operator from $L^2(\partial B_\Gd(\B0))$ to $\Xi$.
\end{lemma}

Let us introduce further the adjoint operator of $K$, i.e., the operator
$K^*:\Xi\rightarrow L^2(\partial B_\Gd)$ defined through the relation,
\begin{equation}\label{EQ:K*}
    (Kv,u)_{\Xi}=(v,K^*u)_{L^2(\partial B_\Gd(\B0))},\ \forall u\in\Xi, v\in L^2(\partial B_\Gd(\B0))
\end{equation}
where $(\cdot,\cdot)_\Xi$ is the scalar product on $\Xi$ defined in~\eq{EQ:IP}
and  $(\cdot,\cdot)_{L^2(\partial B_\Gd(\B0))}$ denotes the usual scalar product
in $L^2(\partial B_\Gd(\B0))$. We check, by simple change of variables and algebraic
manipulations, that the adjoint operator $K^*$ is given by,
\beq
\displaystyle K^*u(\Bx)=\int_{\partial
B_{a'}(\Bx_0)}u_1(\By)\frac{\partial \Phi(\Bx,\By)}{\partial
\BGv_{\Bx}}ds_{\By}+\int_{\partial B_{R'}(\B0)}u_2(\By)\frac{\partial
\Phi(\Bx,\By)}{\partial \BGv_{\Bx}}ds_{\By}
\eeq{10}
for any $u=(u_1,u_2)\in\Xi$ and $\Bx\in\partial B_{\Gd}(\B0)$, with
$\displaystyle \BGv_{\Bx}=\frac{\Bx}{|\Bx|}=\frac{\Bx}{\Gd}$.

From the compactness and linearity of $K$ as given in Lemma~\ref{lemma-2},
we conclude that the adjoint operator $K^*$ is compact as well.
Furthermore, let us denote by $\Ker(K^*)$ the kernel (i.e., null space) of $K^*$.
Then we have the following result.
\begin{Pro}\label{Pro-1}
 If $\psi=(\psi_1,\psi_2)\in \Ker(K^*)$ then $\psi\equiv(0,0)$ in $\Xi$.
\end{Pro}
\begin{proof}
Let $\psi\in \Ker(K^*)$ and define \beq w(\Bx)=\int_{\partial
B_{a'}(\Bx_0)}\psi_1(\By)\Phi(\Bx,\By)ds_{\By}+\int_{\partial
B_{R'}(\B0)}\psi_2(\By)\Phi(\Bx,\By)ds_{\By}, \mbox{ for
}\Bx\in\RR^d \eeq{11} where the integrals exist as improper
integrals for $\Bx\in \partial B_{a'}(\Bx_0)\cup\partial
B_{R'}(\B0)$. From $K^*\psi=0$ and \eq{10} we have that $w$
satisfies the Laplace equation \beq \left\{\vspace{0.15cm}
\begin{array}{ll}\GD w=0, & \mbox{ in }
B_\Gd(\B0) \vspace{0.15cm}\\
\displaystyle\frac{\partial w}{\partial\BGv_{\Bx}}=0,& \mbox{ on
}\partial B_\Gd(\B0)
\end{array}\right.
\eeq{12}
We then conclude that
\beq w=\mbox{constant} \mbox{ in }B_\Gd(\B0)
\eeq{12'}
We denote this constant by $L$, i.e., $w=\mbox{L}$ in $B_\Gd(\B0)$.
Then, because by definition $w$ is harmonic in $B_{R'}(\B0)\setminus{\bar{B}}_{a'}(\Bx_0)$, from the unique continuation principle, we conclude that
\beq
w=L \mbox{ in } B_{R'}(\B0)\setminus{\bar{B}}_{a'}(\Bx_0)
\eeq{14}
The next relations for $w$ are in fact the classical jump conditions
for the single layer potentials with $L^2$ densities (see
\cite{CoKr-Book98} and references therein). We have,
\beqa
&&\displaystyle\lim_{h\rightarrow +0}\int_{\partial
B_{a'}(\Bx_0)}|w(\Bx\pm
h\BGv_{\Bx})-w(\Bx)|^2ds_{\Bx}=0\label{15-1}\\
&&\nonumber\\
&&\displaystyle\lim_{h\rightarrow +0}\int_{\partial
B_{R'}(\B0)}|w(\Bx\pm h\BGv_{\Bx})-w(\Bx)|^2ds_{\Bx}=0 \label{15-2}\\
&&\nonumber\\
&&\displaystyle\lim_{h\rightarrow +0}\int_{\partial
B_{a'}(\Bx_0)}\left|2\frac{\partial w}{\partial\BGv_{\Bx}}(\Bx\pm
h\BGv_{\Bx})-2\frac{\partial
w}{\partial\BGv_{\Bx}}(\Bx)\pm\psi_1(\Bx)\right|^2ds_{\Bx}=0\label{15-3}\\
&&\nonumber\\
&&\displaystyle\lim_{h\rightarrow +0}\int_{\partial
B_{R'}(\B0)}\left|2\frac{\partial w}{\partial\BGv_{\Bx}}(\Bx\pm
h\BGv_{\Bx})-2\frac{\partial
w}{\partial\BGv_{\Bx}}(\Bx)\pm\psi_2(\Bx)\right|^2ds_{\Bx}=0
\eeqa{15-4} where $\BGv_\Bx=\BGv(\Bx)$ denotes the exterior normal
to $\partial B_{R'}(\B0)$ and $B_{a'}(\Bx_0)$ respectively and all
the integral of the normal derivatives of $w$ exists as improper
integrals. From \eq{14}, \eq{15-1} and \eq{15-2} we obtain that
\beq w=L\mbox{ on } \partial B_{R'}(\B0)\cup\partial B_{a'}(\Bx_0)
\eeq{16} Next note that by definition $w$ is harmonic in
$B_{a'}(\Bx_0)$. Then, uniqueness of the interior Dirichlet
problem for $w$ on $B_{a'}(\Bx_0)$ and \eq{16} implies \beq w=L
\mbox{ in } \bar{B}_{a'}(\Bx_0) \eeq{17} From \eq{14}, \eq{17},
and the two jump relations \eq{15-3}, we obtain that \beq \psi_1=0
\mbox{ on } \partial B_{a'}(\Bx_0) \eeq{17'} Equation \eq{17'}
used in the definition of $w$ given at \eq{11}, implies \beq
w(\Bx)=\int_{\partial
B_{R'}(\B0)}\psi_2(\By)\Phi(\Bx,\By)ds_{\By}, \mbox{ for
}\Bx\in\RR^d \eeq{17''} Next, relations \eq{14}, \eq{16}, and
\eq{17} imply that \beq
    w=L \mbox{ in } {\bar B}_{R'}(\B0) .
\eeq{18}
Let us now observe that Green's theorem applied to $w$ in
$ B_{R'}(\B0)$ gives
\beq
    \int_{\partial B_{R'}(\B0)}\frac{\partial w}{\partial \BGv_\Bx} ds_\Bx=0.
\eeq{19}
On the other hand, from the interior jump condition given in \eq{15-4} together with \eq{18}
we have that
\beq
\displaystyle \frac{\partial w}{\partial \BGv_\Bx} = -\frac{1}{2}\psi_2 \mbox{ a.e. on }
\partial B_{R'}(\B0)
\eeq{20} From \eq{19} and \eq{20} we deduce \beq \displaystyle
\int_{\partial B_{R'}(\B0)}\psi_2(\Bx)ds_\Bx=0 . \eeq{21} Observe
that \eq{21} guarantees the bounded behavior of $w$ at infinity in
two dimensions while it is well known that $w$ will decay to zero
at infinity in three dimensions. Then, the classical
representation result for smooth functions, which are harmonic in
the exterior of a given smooth region and bounded at infinity (see
\cite{CoKr-Book98}), implies \beq
    w(\Bx)=w_{\infty}+\int_{\partial B_{R'}(\B0)}\left(w(\By)\frac{\partial \Phi(\Bx,\By)}{\partial\BGv_\By}-\frac{\partial
w}{\partial\BGv_\By}(\By) \Phi(\Bx,\By)\right)ds_{\By}
\eeq{22}
for all $\Bx\in\RR^d\setminus B_{R'}(\B0)$ and for some constant $w_\infty$
which depends only on the dimension. Using \eq{16} and \eq{20} in
\eq{22} we obtain
\beqa
w(\Bx)&=&w_\infty + L\int_{\partial B_{R'}(\B0)}\frac{\partial \Phi(\Bx,\By)}{\partial\BGv_\By}ds_{\By}+
\frac{1}{2}\int_{\partial B_{R'}(\B0)}\psi_2(\By)\Phi(\Bx,\By)ds_{\By}\nonumber\\
&&\nonumber\\
&=& w_\infty+\frac{1}{2}w(\Bx)\nonumber\\
&&\nonumber\\
&=& 2w_\infty \mbox{ for }\Bx\in\RR^d\setminus {\overline
B}_{R'}(\B0) \eeqa{23} where we used \eq{11} for the last integral
in the first line of \eq{23}. Finally, \eq{18} and \eq{23}
together with the pair of jump conditions given at \eq{15-4} imply
that \beq \psi_2=0 \mbox{ a.e. on }
\partial B_{R'}(\B0) . \eeq{24} The statement of the Proposition
follows from \eq{17'} and \eq{24}.
\end{proof}

Before presenting the main result of this work, let us introduce the following space of functions
\[
    U \equiv K(C(\partial B_\Gd(\B0)) .
\]
It is clear that $U$ is a subspace of $\Xi$. Moreover we have,
\begin{lemma}\label{LMMA:U}
    The set $U\subset\Xi$ is dense in $\Xi$.
\end{lemma}
\begin{proof}
    We first observe that the subspace $U\subset \Xi$ satisfies
\beq
    \overline{U} =\left(U^\perp\right)^\perp
\eeq{26} where here and further in the proof, for a given set
$M\subset\Xi$, $\overline{M}$ and $M^\perp$ denote its closure and
orthogonal complement respectively in the $L^2$ topology generated
on $\Xi$ by the scalar product defined at \eq{EQ:IP}. Property
\eq{26} is classic for subspaces in a Hilbert space (see
\cite{B}). On the other hand we also have that \beq
    U^\perp=\Ker(K^*)
\eeq{27}
Indeed let $\xi=(\xi_1,\xi_2)\in U^\perp$. Then, for all $\ph\in C(\partial B_\Gd(\B0))$  we have,
\beqa
    0=(K\ph,\xi)_{\Xi}&\Leftrightarrow&(\ph,K^*\xi)_{L^2(\partial B_\Gd(\B0))}=0\Leftrightarrow\nonumber\\
&\Leftrightarrow& K^*\xi=0\Leftrightarrow \xi\in
\Ker(K^*) .
\eeqa{28}
Properties \eq{26} and \eq{27} imply that
\beq
{\overline U}= \Ker(K^*)^\perp
\eeq{29}
Proposition \ref{Pro-1} together with \eq{29} imply the density of $U$ in $\Xi$.
\end{proof}

We are now in the position to state and prove the main result of the paper.
\begin{Thm}
\label{thm-1}
 Let $a,c,a',R',R$ be given as in (\ref{par-cond}).
 Let $v=(v_1,v_2)\in C({\bar B}_{a'}(\Bx_0))\times C(\RR^d\setminus B_{R'}(\B0))$
 be such that $v_1$ is harmonic in $B_{a'}(\Bx_0)$ and $v_2$ is harmonic in $\RR^d\setminus {\bar B}_{R'}(\B0)$.
Define the double layer potential $\CD$ with density $\varphi\in L^2(\partial B_\Gd(\B0))$ as,
 $$ \displaystyle\CD \varphi(\Bx)=\int_{\partial
 B_\Gd(\B0)}\varphi(\By)\frac{\partial\Phi(\Bx,\By)}{\partial\nu_{\By}}ds_\By, \mbox{ for }\Bx\in \RR^d\setminus \bar{B}_\Gd(\B0)$$
 Then $\CD:L^2(\partial B_\Gd(\B0))\rightarrow C(\RR^d\setminus {\overline B}_{a'-a+\Gd}(\B0))$
 is a continuous operator between $L^2(\partial B_\Gd(\B0))$ and $C(\RR^d\setminus
 {\overline B}_{a'-a+\Gd}(\B0))$ endowed with their natural topologies. Moreover, there exists a sequence $\{v_n\}\subset C(\partial B_\Gd(\B0))$ such that
 $$\CD v_n\rightarrow v_1 \mbox{ strongly in }C(\bar{B}_a(\Bx_0)), \mbox {and }, \CD v_n\rightarrow v_2 \mbox{ strongly in } C(\RR^d\setminus
 B_R(\B0))$$ with respect to the uniform topology of $C(\bar{B}_a(\Bx_0))$ and $C(\RR^d\setminus
 B_R(\B0))$.
\end{Thm}
\begin{proof}
    We first observe that $v\in\Xi$.
%In order to simply the exposition we will introduce $U\subset \Xi$ given by \beq U\doteq K(C(\partial B_\Gd))\eeq{25} We will first show that $U$ is dense in $\Xi$. Note that $U$ is a subspace of $\Xi$ and satisfies \beq {\bar U}=\left(U^\perp\right)^\perp\eeq{26} where here and further in the proof, for a given set $M\subset\Xi$, by ${\bar M}, M^\perp$ we denote its closure and orthogonal complement respectively in the $L^2$ topology generated on $\Xi$ by the scalar product defined at \eq{4}. Property \eq{26} is classic for subspaces in a Hilbert space (see \cite{B}). On the other hand we also have that \beq U^\perp=\Ker(K^*)\eeq{27} Indeed let $\xi=(\xi_1,\xi_2)\in U^\perp$. Then, for all $\ph\in C(\partial B_\Gd)$  we have, \beqa 0=(K\ph,\xi)_{\Xi}&\Leftrightarrow&(\ph,K^*\xi)_{L^2(\partial B_\Gd)}=0\Leftrightarrow\nonumber\\ &\Leftrightarrow& K^*\xi=0\Leftrightarrow \xi\in \Ker(K^*)\eeqa{28} and \eq{27} is proved. Properties \eq{26} and \eq{27} imply \beq {\bar U}= \Ker(K^*)^\perp\eeq{29} Proposition \ref{Pro-1} together with \eq{29} imply the density of $U$ in $\Xi$.
Then the definition of $U$ and Lemma~\ref{LMMA:U} imply that there exists a sequence $\{v_n\}\subset C(\partial B_\Gd(\B0))$ such that
\beq
K(v_{n})\rightarrow v \mbox{ strongly in }\Xi
\eeq{29'}
From the definition of the $\Xi$ topology and \eq{29'} we conclude that
\beqa
&&\Vert K_1v_{n}-v_1\Vert_{L^2(\partial B_{a'}(\Bx_0))}\rightarrow
 0\nonumber\\
 &&\label{30'}\\
&&\Vert K_2v_{n}-v_2\Vert_{L^2(\partial B_{R'}(\B0))}\rightarrow
 0\nonumber
\eeqa{30}
Observe that, by definition, $K_1v_{n}$ (resp. $K_2 v_{n}$) is the restriction
to $\partial B_{a'}(\Bx_0)$ (resp. $\partial B_{R'}(\B0)$) of $\CD v_{n}$ (resp. $\CD v_{n}$)
where $\CD$ was defined in the statement of the Theorem. From the properties of $\CD$,
the hypothesis on $v_1,v_2$ and the regularity results of
Lemma~\ref{lemma-1} we conclude that
\beqa &&\Vert
\CD v_{n}-v_1\Vert_{C(\bar{B}_{a}(\Bx_0))}\leq C_1\Vert K_1 v_{n}-v_1\Vert_{L^2(\partial B_{a'}(\Bx_0))}\nonumber\\
 &&\label{31}\\
&&\Vert \CD v_{n}-v_2\Vert_{C(\RR^d\setminus B_{R}(\B0))}\leq
C_2\Vert K_2v_n-v_2\Vert_{L^2(\partial
B_{R'}(\B0))}\nonumber\eeqa{31'} where we have also used the
properties of $a, a', R, R'$ and $\Bx_0$ stated at
(\ref{par-cond}). Finally from \eq{30'} and \eq{31} we obtain the
statement of the Theorem.
\end{proof}

%%%%%%%%%%%%%%%%%%%%%%%%%%%%%%%%%%%%%%%%%%%%%%%%%%%%%%%%%%%%%%%%%%
%%%%%%%%%%%%%%%%%%%%%%%%%%%%%%%%%%%%%%%%%%%%%%%%%%%%%%%%%%%%%%%%%%
\section{The minimal energy solution}
\label{SEC:MES}
%%%%%%%%%%%%%%%%%%%%%%%%%%%%%%%%%%%%%%%%%%%%%%%%%%%%%%%%%%%%%%%%%%
%%%%%%%%%%%%%%%%%%%%%%%%%%%%%%%%%%%%%%%%%%%%%%%%%%%%%%%%%%%%%%%%%%

Theorem \ref{thm-1} implies that there exist infinitely many
functions $h\in C(\partial B_\Gd(\B0))$ (resp. $g\in C(\partial
B_\Gd(\B0))$) as solutions to~\eq{2} in Formulation A
((resp.~\eqref{3}) in Formulation A$'$). Indeed, let $0<\eps\ll 1$
and $v=(v_1,v_2)\in C({\bar B}_{a'}(\Bx_0))\times C(\RR^d\setminus
B_{R'}(\B0))$ such that $v_1$ is harmonic in $B_{a'}(\Bx_0)$ and
$v_2$ is harmonic in $\RR^d\setminus {\bar B}_{R'}(\B0)$. Then,
using the regularity results of Lemma \ref{lemma-1} we observe
that any function $h\in C(\partial B_\Gd(\B0))$ satisfying
\begin{equation}\label{EQ:Approx-1}
\|Kh-v\|_\Xi \le \eps,
\end{equation}
where the $\|\cdot\|_\Xi$ is the natural norm induced by the inner
product defined in~\eqref{EQ:IP}, must be a solution for the
problem \eq{2}. This together with \eq{30} provides a sequence of
solutions for problem \eq{2}.

Now, we will prove how, for any desired level of accuracy $\eps$,
among the solutions of~\eqref{EQ:Approx-1}, there exists a unique
solution with minimal energy norm, i.e., with minimal
$L^2(\partial B_\delta(\B0))$ norm. We have the following result.
\begin{corollary}\label{cor-0}
Let $0<\eps\ll 1$ and $v\in \Xi$ be given. Then there exists a unique $h_0 \in L^2(\partial B_\Gd(\B0))$
solution of the following minimization problem,
\begin{equation}\label{min-1}
    \|h_0\|_{L^2(\partial B_\Gd(\B0))}=\min_{\|Kh-v\|_{\Xi}\le \eps} \|h\|_{L^2(\partial B_\Gd(\B0))}
\end{equation}
\end{corollary}
\begin{proof}
From Proposition \ref{Pro-1} and classical linear operator theory
we have that the linear bounded operator $K: L^2(\partial
B_\Gd(\B0))\rightarrow \Xi$ has a dense range. This together with
the classical theory of minimum norm solutions based on the
Tikhonov regularization implies the statement of the Corollary
(see \cite{Kress-Book99}, Theorem 16.12). In fact the classical
theory implies that the solution $h_0$ of (\ref{min-1}) belongs
$C(\partial B_\Gd(\B0))$ and is the unique solution of
\begin{equation}\label{EQ:E-L}
    \alpha h_\alpha+K^*K h_\alpha=K^*v, \mbox{ with } \|Kh_\alpha-v \|_{\Xi}=\eps,
\end{equation}
as the regularization strength $\alpha$ goes to $0$.
\end{proof}

The next result is an immediate consequence of Theorem \ref{thm-1}. It proves the existence of a class of solutions
for the problem \eq{3}.
\begin{corollary}\label{cor-1}
 Let $u_0$ and $u_1$ be as in \eq{2} and consider $v=(u_1-u_0,0)$. Then there exist infinitely many functions
 $g\in C(\partial B_\Gd(\B0))$ such that ${\cal D}g=u$ with $u$ satisfying~\eq{3}. Moreover, there exists a unique
 function $g\in C(\partial B_\Gd(\B0))$ solution of \eq{3} with minimal  $L^2(\partial B_\Gd(\B0))$ norm.
\end{corollary}
\begin{proof}
First observe that $v=(u_1-u_0,0)$ satisfies the hypothesis of
Theorem \ref{thm-1} for $a'>a$ satisfying (\ref{par-cond}) and
small enough so that $u_1$ remains harmonic on $B_{a'}(\Bx_0)$.
Thus we have that there exists a sequence
 $\{g_n\}\subset C(\partial B_\Gd(\B0))$ such that we have
\beqa &&\Vert \CD g_n+u_0\Vert_{C(\bar{B}_{a}(\Bx_0))}\rightarrow
 0\nonumber\\
 &&\label{32}\\
&&\Vert \CD g_n\Vert_{C(\RR^d\setminus B_{R}(\B0))}\rightarrow
 0\nonumber\eeqa{32'}
 where $\CD$ is as in Theorem \ref{thm-1}. Then, for $0<\eps\ll 1$ as in \eq{3}, we can choose $N$ such that for all $n\ge N$ we will have
 \beqa&&\Vert
\CD g_n+u_0\Vert_{C(\bar{B}_{a}(\Bx_0))}\leq \epsilon\nonumber\\
 &&\label{33}\\
&&\Vert \CD g_n\Vert_{C(\RR^d\setminus
B_{R}(\B0))}\leq\epsilon\nonumber \eeqa{33'} This implies that,
there exists an index $N$ such that for all $n\ge N$, functions of
the form $\CD g_n$ will be solutions of the problem \eq{3}. Next,
by using Corollary~\ref{cor-0} we obtain the existence of a
solution $g\in C(\partial B_\Gd(\B0))$ with minimal $L^2(\partial
B_\Gd(\B0))$ norm.
\end{proof}

\begin{Arem}\label{Rem1}
We observe that one can easily adapt the proof of Theorem
\ref{thm-1} to the general case stated in Question
~\ref{QUST:Main}, i.e., the case of finitely many mutually
disjoint balls of interest. Thus, following the same arguments
as before, one will obtain a class of solutions for
Question~\ref{QUST:Main} in this general context. Moreover, by
adapting the proof of Corollary~\ref{cor-0} to the general case of
of $N$ disjoint domains we could obtain the existence of a minimal
$L^2(\partial B_\Gd)$- norm solution for the problem.
\end{Arem}

\begin{Arem} We also mention that all the results 
in this paper readily extend to general simple connected domains with $C^2$ boundary but 
for the clarity of the exposition we chose to present the results only 
in the case of spherically shaped domains.

\end{Arem}

\begin{Arem}\label{Rem3}
It is well-known that both the interior and exterior Dirichlet
problems are stable with respect to boundary data. This means that
small perturbations on boundary data $h$ produce small
perturbations in the solution $u$ in Formulation A. This further
suggests that the inverse problem that we consider in this work,
however, is unstable. To find the source function $h$ that
generate desired field $v$, we need to invert a compact integral
operator. Such a problem is always an ill-posed
problem~\cite[Theorem 1.17]{Kirsch-Book96} and this is the main
reason for the consideration of minimal energy solution.
\end{Arem}

%Several other interesting applications are possible using our general strategy for the manipulation of fields.
%\begin{Arem}\label{Rem2}
%In the general context of Question~\ref{QUST:Main}, Theorem~\ref{thm-1} implies that in fact the active source (antenna) at $B_\Gd$ can be programmed to approximate a zero field in $\displaystyle\bigcup_{i=1}^N R_i$ and
%any desired scattering field in $R_0$, for example the field corresponding to a very large object, thus creating an illusion for an external observer. In a different application, one can program the active source (antenna) to approximate $N$ different desired fields in each of the regions $R_i$ for $i\in\{1,...,N\}$ while creating a zero field region in $R_0$ thus sending information in the regions of interests without being detected by an outside observer.
%\end{Arem}

%%%%%%%%%%%%%%%%%%%%%%%%%%%%%%%%%%%%%%%%%%%%%%%%%%%%%%%%%%%%%%%%%%
%%%%%%%%%%%%%%%%%%%%%%%%%%%%%%%%%%%%%%%%%%%%%%%%%%%%%%%%%%%%%%%%%%
\section{Numerical simulations}
\label{SEC:Num}
%%%%%%%%%%%%%%%%%%%%%%%%%%%%%%%%%%%%%%%%%%%%%%%%%%%%%%%%%%%%%%%%%%
%%%%%%%%%%%%%%%%%%%%%%%%%%%%%%%%%%%%%%%%%%%%%%%%%%%%%%%%%%%%%%%%%%

We now present some numerical results to demonstrate the ideas that we have developed. We consider both two-dimensional and three-dimensional cases. To simplify the visualization, we only present results with regions of interests being balls, although the numerical algorithms we developed can deal with regions of arbitrary shapes with boundaries regular enough. The scattering problem (more precisely, the integral operator $K$) is discretized by the Nystr\"om method, following the presentation in~\cite{CoKr-Book98}.

In the two-dimensional case, we consider Question~\ref{QUST:Main} with $N=2$, $\Gd=1$, 
$D_1=B_2(\Bx_1)$, $D_2=B_2(\Bx_2)$ and $D=B_{15}(\B0)$. The centers of $D_1$ and $D_2$ are $\Bx_1=(0,12)$ and $\Bx_2=(10,0)$ respectively. The fields are $u_1=\log\frac{1}{|\Bx|}$, $u_2=\frac{x}{|\Bx|^2}$ and $u_0=0$.
The accuracy parameter is $\eps=10^{-3}(\|u_1\|_{L^2(D_1)}+\|u_2\|_{L^2(D_2)}+\|u_0\|_{L^2(\partial D)})$. We show in the left plot of Fig.~\ref{FIG:2D} the minimal energy solution of the problem with the desired fields
given as above. The source function, supported on the unit circle,
is parameterized using the azimuth angle $\varphi\in[0,2\pi)$. The two middle plots of Fig.~\ref{FIG:2D} show the relative differences of the field that is generated by the minimal energy solution and the desired
field in region $D_1$ and $D_2$. It is clear from the plot that the solution
strategy works almost perfectly because the mismatch between the
desired field and the generated field is almost very small everywhere.
\begin{figure}[ht]
    \centering
    \includegraphics[trim=0cm 0cm 0cm 0cm,clip,angle=0,width=0.3\textwidth]{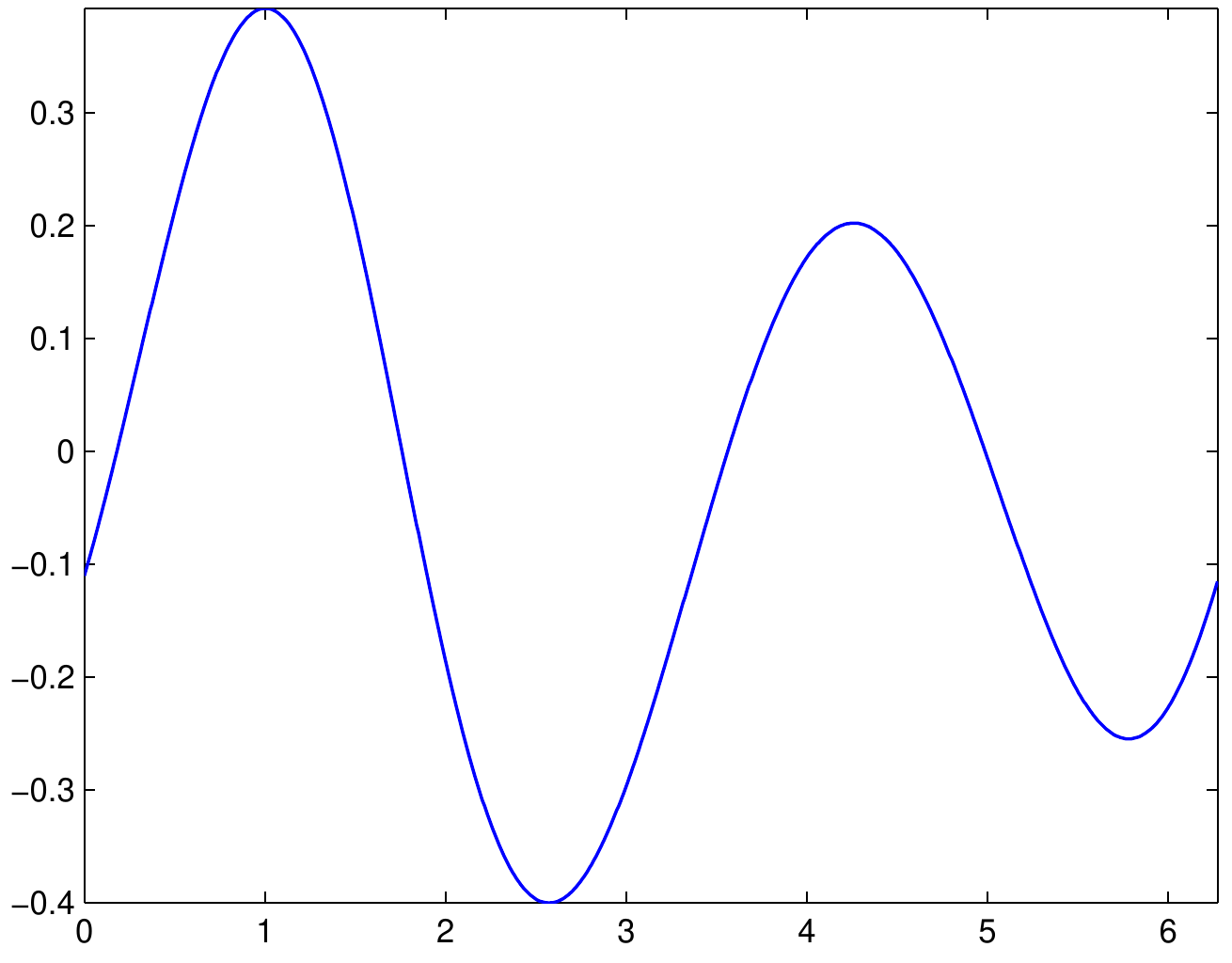}
    \includegraphics[trim=0cm 0cm 0cm 0cm,clip,angle=0,width=0.3\textwidth]{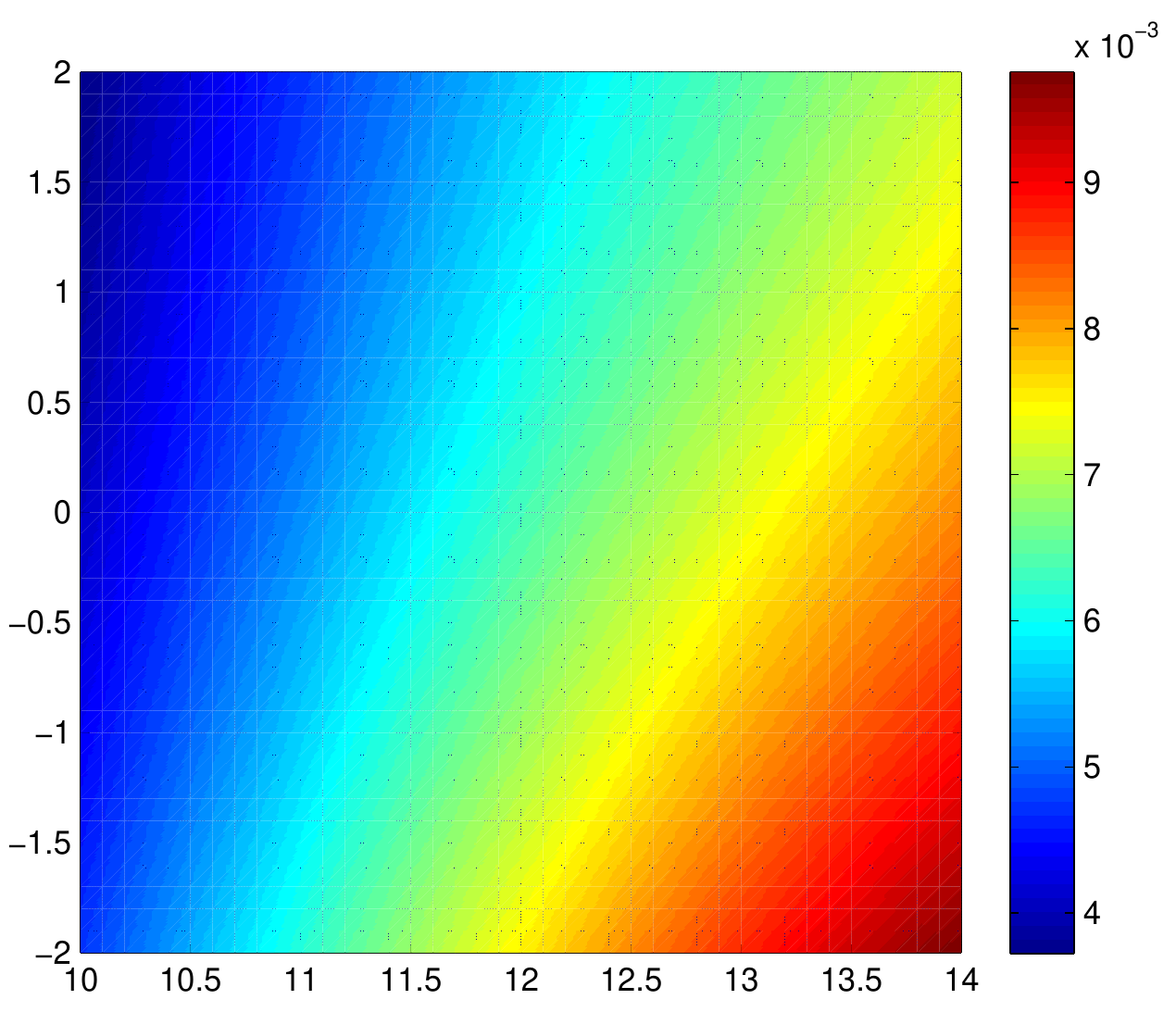}\\
    \includegraphics[trim=0cm 0cm 0cm 0cm,clip,angle=0,width=0.3\textwidth]{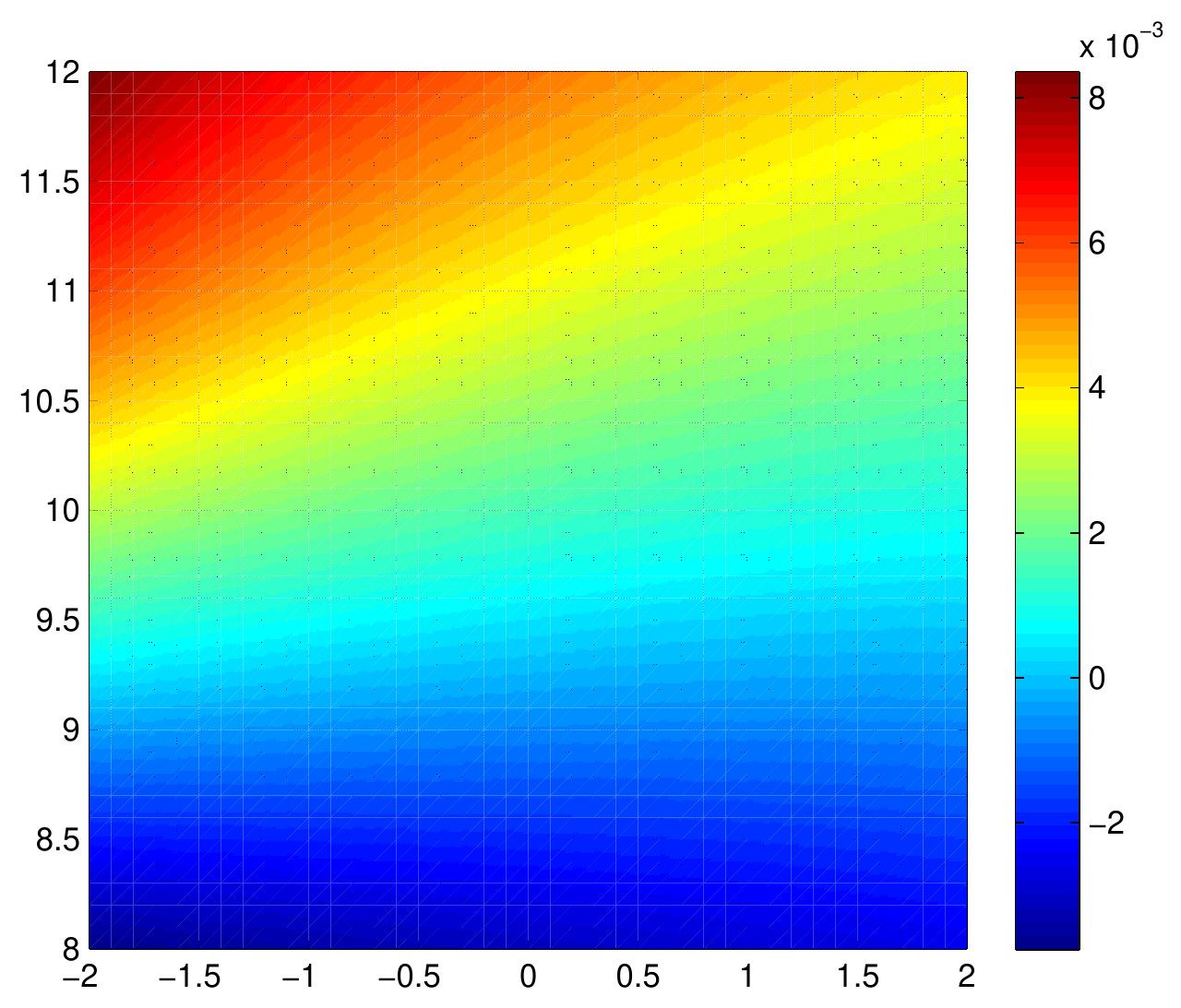}
    \includegraphics[trim=0cm 0cm 0cm 0cm,clip,angle=0,width=0.3\textwidth]{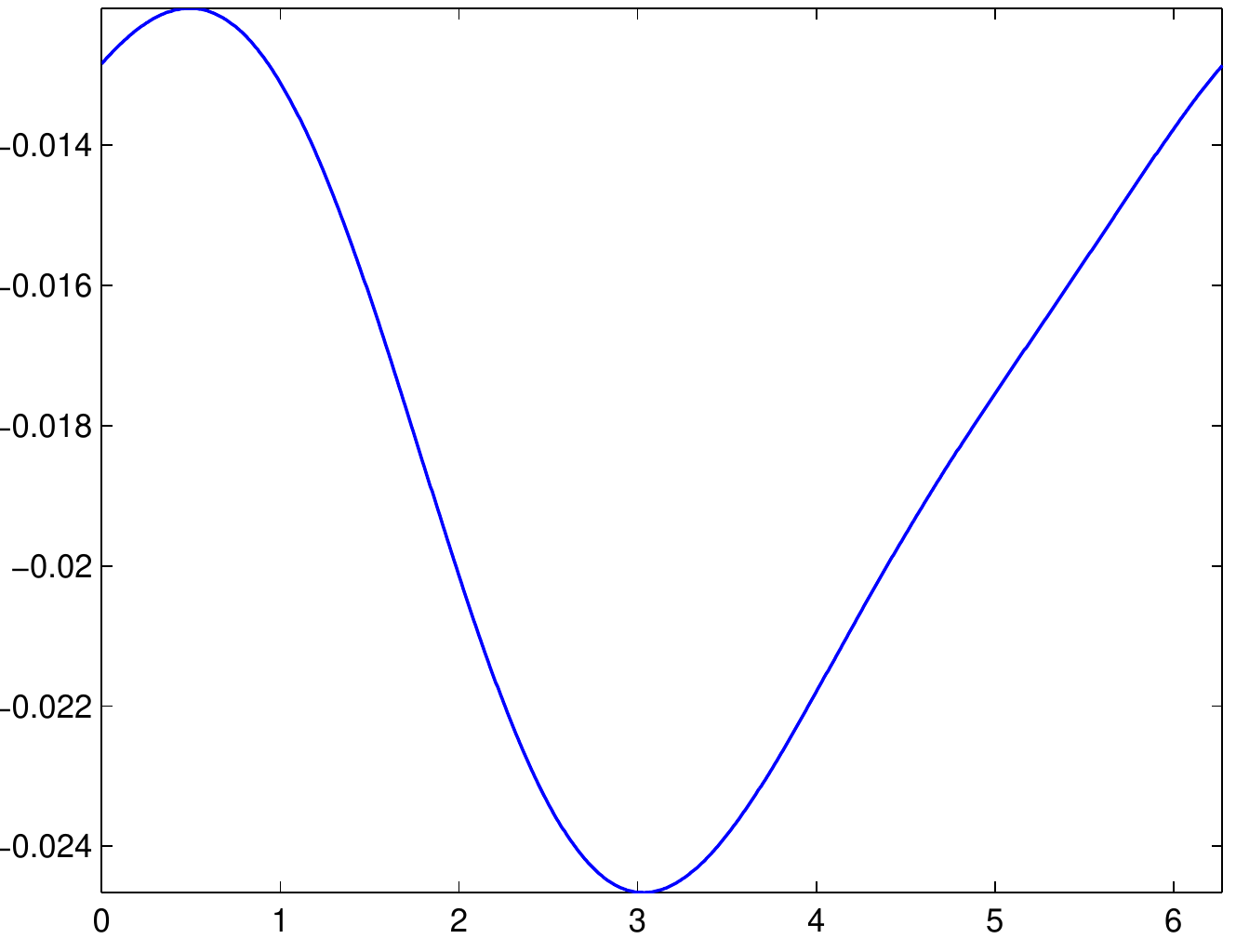}
    \caption{Numerical results in the two-dimensional case. From top left to bottom right: the minimal energy source function $f(\varphi)$, the relative difference between generated and desired fields in the neighborhoods of region $D_1$, $D_2$, generated field on $\partial D$.}
    \label{FIG:2D}
\end{figure}

In the three-dimensional case, we observe very similar results.
Given an arbitrary point $\Bx_1=(10,0,0)$ we considered Question
~\ref{QUST:Main} with $N=1$, $\Gd=1$, $u_0=0$, $u_1=\frac{1}{|\Bx|}$, $D_1=B_2(\Bx_1)$ and $D=B_{15}(\B0)$. The
accuracy parameter is again $\eps=10^{-3}(\|u_1\|_{L^2(D_1)}+\|u_0\|_{L^2(\partial D)})$. The results are shown
in Fig.~\ref{FIG:3D}. On the left plot, we show the minimal energy
source $h(\theta,\varphi)$ where the unit ball is parameterized
using the polar angle $\theta\in[-\pi/2,\pi/2]$ and the azimuth
angle $\varphi\in[0,2\pi)$. On the middle plot, we show the
difference between the generated field and the desired filed on
$\partial B_2(\Bx_1)$. The right plot shows the difference between
the generated field and the desired filed on $\partial B_{15}(\B0)$.
Due to the limitations of visualization, we are not able to show
the difference inside the balls which we observe to be small.
\begin{figure}[ht]
    \centering
    \includegraphics[trim=0cm 0cm 0cm 0cm,clip,angle=0,width=0.3\textwidth]{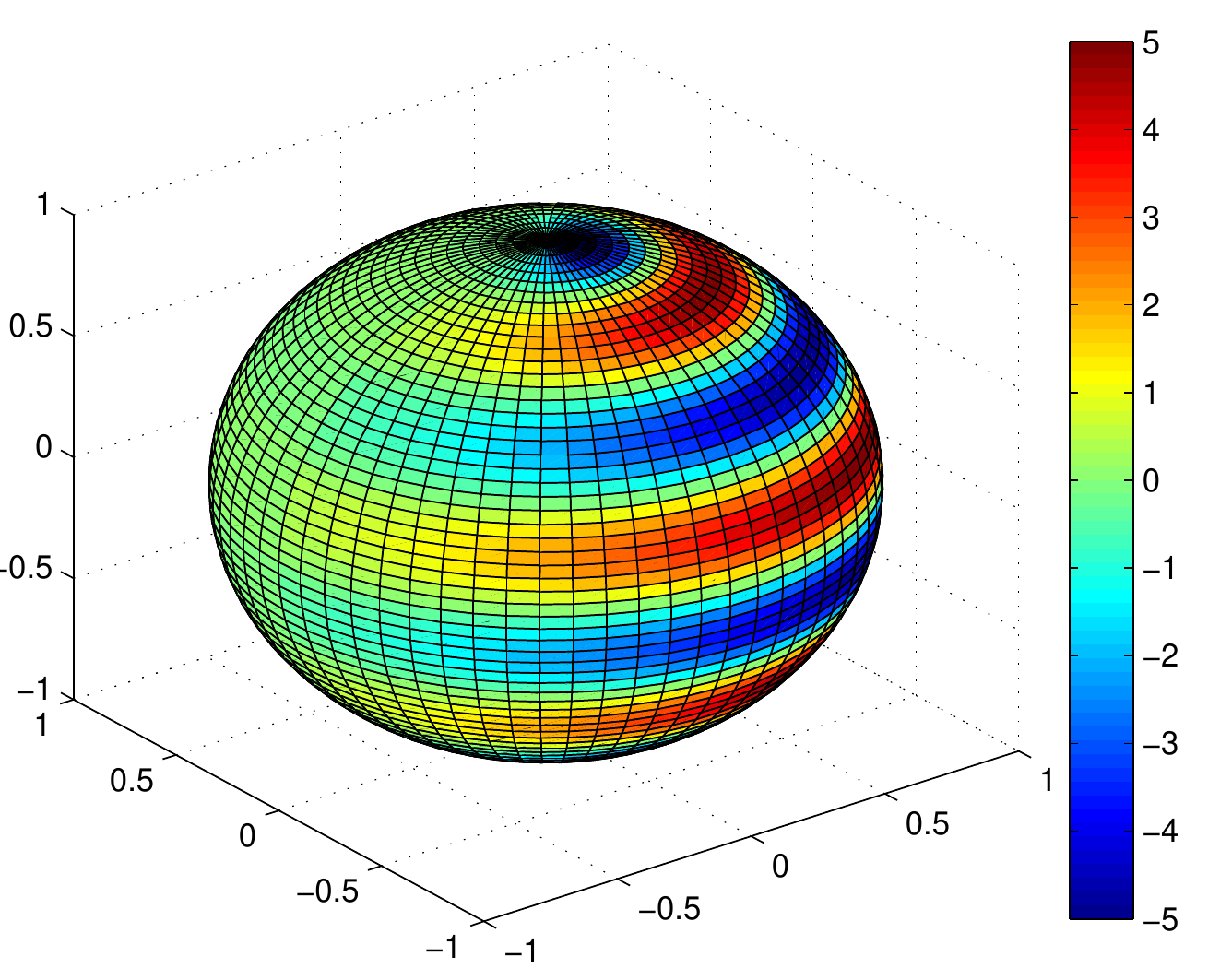}
    \includegraphics[trim=0cm 0cm 0cm 0cm,clip,angle=0,width=0.3\textwidth]{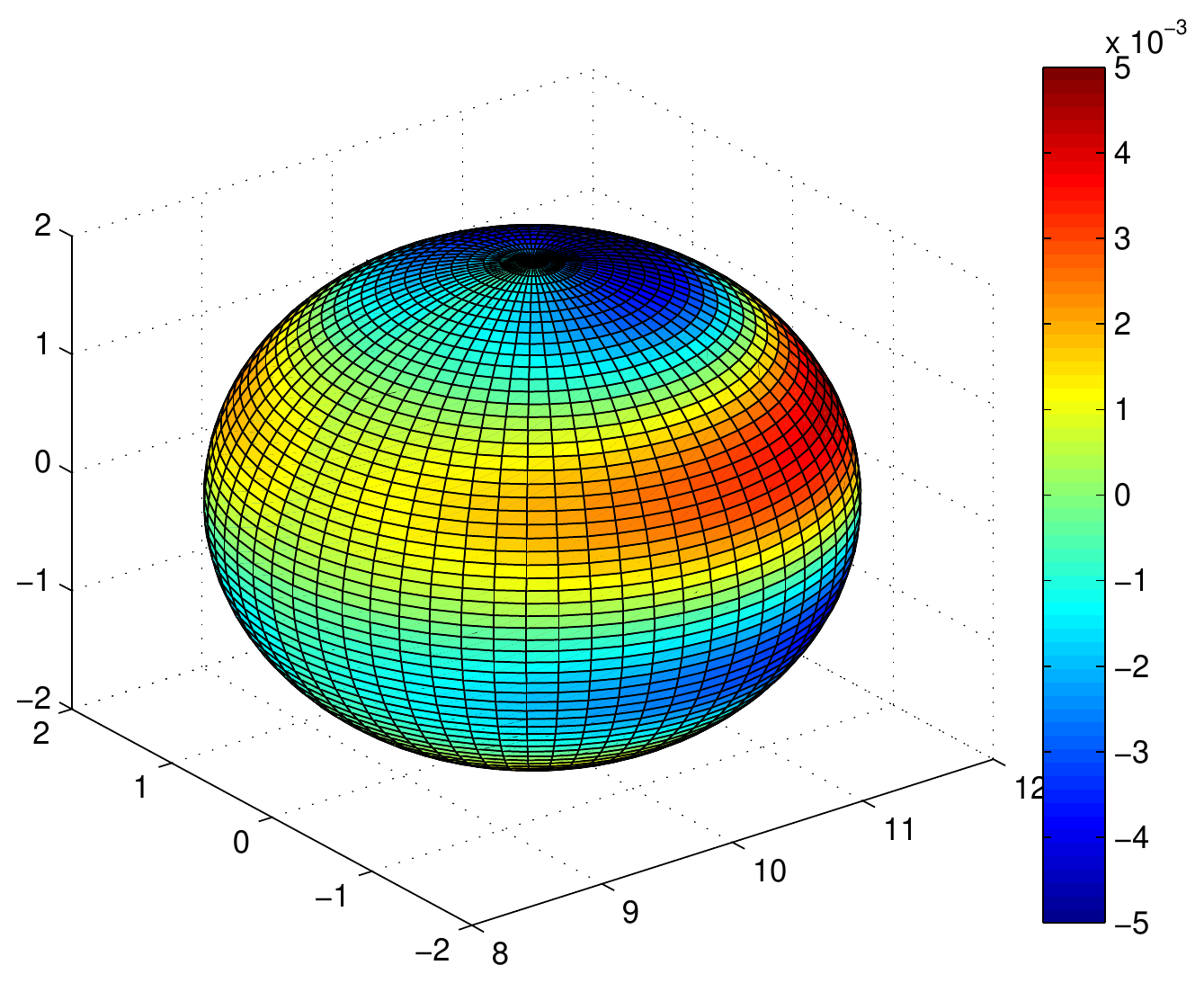}
    \includegraphics[trim=0cm 0cm 0cm 0cm,clip,angle=0,width=0.3\textwidth]{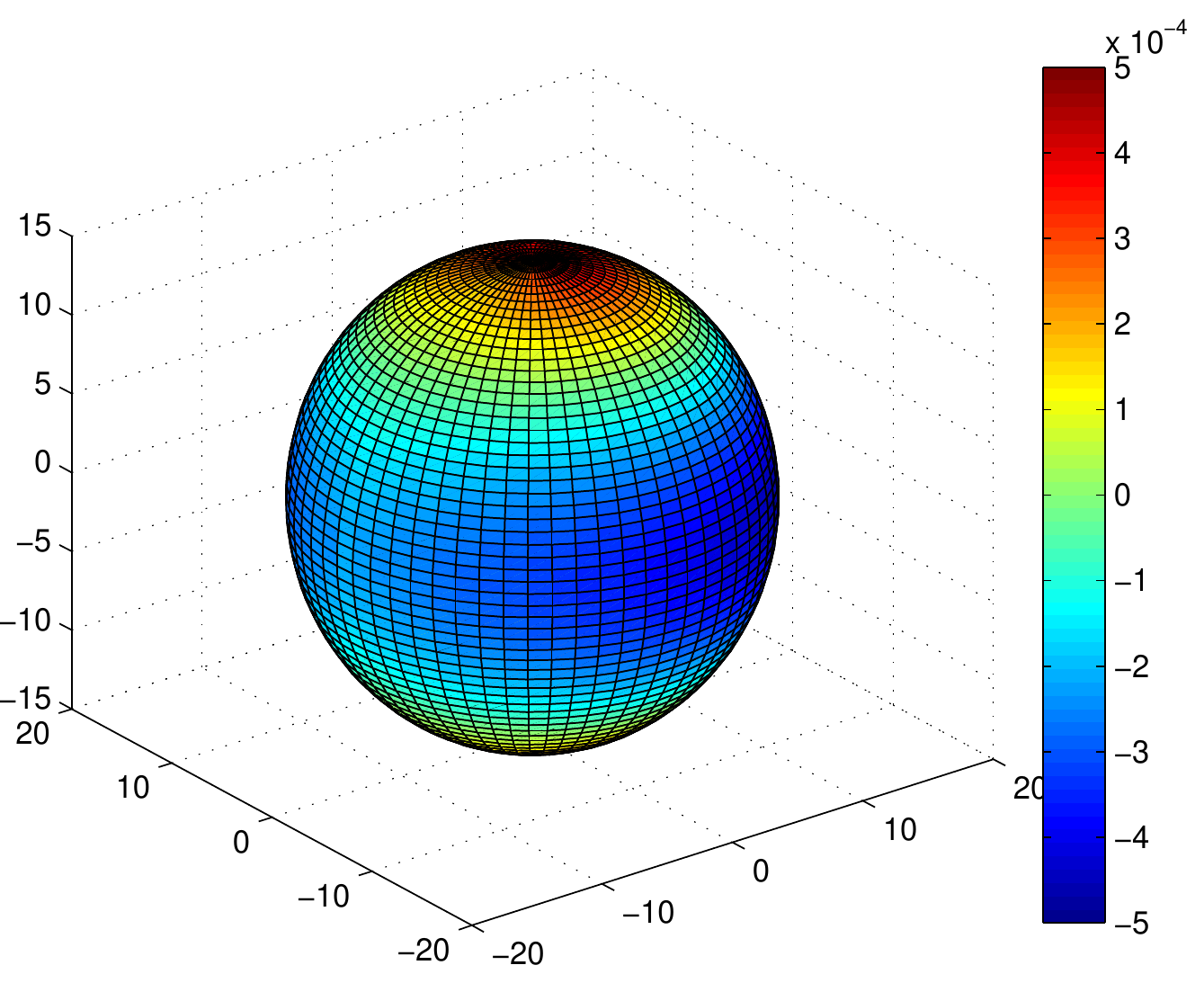}
    \caption{Numerical results in the three-dimensional case. Left: plot of the minimal energy source function $f(\theta,\varphi)$; Middle: plot of the relative difference field on $\partial B_2(\Bx_1)$; Right: plot of the difference field on $\partial B_{15}(\B0)$.}
    \label{FIG:3D}
\end{figure}

The numerical simulations support our proof that the strategy
proposed in this work on generating desired fields in different
regions works well. More systematic numerical studies of the
problem under various situations, together with a stability
analysis will be reported in ~\cite{RenOno}.

%%%%%%%%%%%%%%%%%%%%%%%%%%%%%%%%%%%%%%%%%%%%%%%%%%%%%%%%%%%%%%%%%%
%%%%%%%%%%%%%%%%%%%%%%%%%%%%%%%%%%%%%%%%%%%%%%%%%%%%%%%%%%%%%%%%%%
\section{Concluding remarks}
\label{SEC:Concl}
%%%%%%%%%%%%%%%%%%%%%%%%%%%%%%%%%%%%%%%%%%%%%%%%%%%%%%%%%%%%%%%%%%
%%%%%%%%%%%%%%%%%%%%%%%%%%%%%%%%%%%%%%%%%%%%%%%%%%%%%%%%%%%%%%%%%%

The idea of manipulating quasistatic fields (or in general acoustic and electromagnetic fields)
to generate desired scattering effects has been explored extensively in the engineering community
recently due to its practical importance. In this work we present a systematic method to analyze
mathematically and numerically the feasibility of the active field manipulation strategy. 
 In the quasistatic regime, we show that one can find source functions that are able to generate desired quasistatic field
 in multiple regions of interests, to any given accuracy. This enable us to use the active source
 to create desired illusions or energy focusing without being detected by observations performed
 outside of the domain of interests. In fact, we show that for any given accuracy, there are infinitely
 many sources that can achieve the same effects. The source function that has the minimal energy is
 probably the one that is physically relevant and is weakly stable. Our numerical simulations
 confirm that the strategy can indeed be realized.

The formulation that we present is independent of the spatial
dimension it provides a first step towards the development of
field manipulation techniques in more complicated settings, such
as in low-to-medium frequency acoustic and TE or TM electromagnetic regimes
even though the analysis in those regimes need to be done
carefully due to the change of the integral kernels and the
presence of resonances.
%The numerical scheme for solving the
%problem remains basically unchanged and all we need to do is to
%change the Green function of the problem.
In addition, if the problem is posed in a non-homogeneous medium,
with \emph{known} medium property, the same formulation can be
constructed and the same type of minimal energy solution can be
obtained through the Euler-Lagrange equation~\eqref{EQ:E-L}.

Another essential discussion is about the stability of the
solution. As it is well known the problem of inverting a compact
operator is highly unstable and that is why we focus on the most
physical relevant solution, namely the unique minimum energy
solution. By using the generalized discrepancy principle it can 
be shown that this solution is $L^2$ stable with
respect to small errors at the antennas or in the measurements of
the right hand side data. The $L^2$ stability analysis together
with the associated numerical discussion for the minimal norm
solution will be presented in ~\cite{Ono9}.

There are many potential applications of the method as we
mentioned in the Introduction. {\bf{Formulation A}} with $u_1=0$
corresponds to the problem of the quasistatic active exterior
cloaking as described in \cite{OMV1}. It has been shown
in~\cite{OMV1,OMV2}(see also \cite{OMV3,OMV4} for acoustics) that
with a few active point sources, one can generate similar effects
as what we propose here. This is not surprising because of the
non-uniqueness nature of the problem. Indeed, we believe that the
cases in~\cite{OMV1,OMV2} are special cases of the current
framework, if we are allowed to use continuous functions to
approximate the delta function model of point sources.
Numerically, this can be done by searching for solutions with
minimal $L^1$ norm instead of the $L^2$ norm (the energy norm).
The numerical techniques of $l^1$ minimization can then be
employed to solve the minimization problem.

%%%%%%%%%%%%%%%%%%%%%%%%%%%%%%%%%%%%%%%%%%%%%%%%%%%%%%%%%%%%%%%%%%
%%%%%%%%%%%%%%%%%%%%%%%%%%%%%%%%%%%%%%%%%%%%%%%%%%%%%%%%%%%%%%%%%%
\section*{Acknowledgment}
%%%%%%%%%%%%%%%%%%%%%%%%%%%%%%%%%%%%%%%%%%%%%%%%%%%%%%%%%%%%%%%%%%
%%%%%%%%%%%%%%%%%%%%%%%%%%%%%%%%%%%%%%%%%%%%%%%%%%%%%%%%%%%%%%%%%%

The author would like to thank Graeme W. Milton for very useful discussions 
and Kui Ren for help with the numerical support of the results. 

%%%%%%%%%%%%%%%%%%%%%%%%%%%%%%%%%%%%%%%%%%%%%%%%%%%%%%%%%%%%%%%%%%
%%%%%%%%%%%%%%%%%%%%%%%%%%%%%%%%%%%%%%%%%%%%%%%%%%%%%%%%%%%%%%%%%%
\section*{Appendix}
\label{appendix}
%%%%%%%%%%%%%%%%%%%%%%%%%%%%%%%%%%%%%%%%%%%%%%%%%%%%%%%%%%%%%%%%%%
%%%%%%%%%%%%%%%%%%%%%%%%%%%%%%%%%%%%%%%%%%%%%%%%%%%%%%%%%%%%%%%%%%

\begin{lemma}\label{lemma-1-A}
Let $0<R_1<R_*<R_2$ be three constants and $\By_0\in\RR^d$ an arbitrary point.
%Consider three concentric balls centered in some arbitrary point, ${\By}_0\in\RR^d$, i.e., $B_{R_1}=
%B_{R_1}({\By}_0)$, $B_{R_*}=B_{R_*}({\By}_0)$, and $B_{R_2}=B_{R_2}({\Bx}_0)$, with $R_1<R_*<R_2$.
Let $f,g\in C(\partial B_{R_*}(\By_0))$ and define $v_i\in C^2(B_{R_*}(\By_0))\cap C^1(\overline{B}_{R_*}(\By_0))$ and
$v_e\in C^2(\RR^d\setminus \overline{B}_{R_*}(\By_0))\cap C^1(\RR^d\setminus B_{R_*}(\By_0))$ to be the solutions of the following interior and exterior Dirichlet problems respectively,
\beq
\left\{\begin{array}{ll}
\GD v_i =0 \mbox{ in } B_{R_*}(\By_0)\vspace{0.15cm}\\
v_i=f \mbox{ on } \partial
B_{R_*}(\By_0)\vspace{0.15cm}\end{array}\right.
\eeq
and
\beq
\left\{\begin{array}{lll}\GD v_e =0 \mbox{ in }\RR^d\setminus \bar
B_{R_*}(\By_0)\vspace{0.15cm}\\
v_e=g \mbox{ on } \partial
B_{R_*}(\By_0)\vspace{0.15cm} \\
v_e=\left\{\begin{array}{ll}O(1)\mbox{ for } |\Bx|\rightarrow
\infty,
\mbox{ if } d=2\vspace{0.15cm}\\
o(1)\mbox{ for } |\Bx|\rightarrow \infty, \mbox{ if }
d=3\end{array}\right.\end{array}\right.
\eeq
Then we have,
$$(i)\; \displaystyle\Vert v_i\Vert_{C(B_{R_1}(\By_0))}\leq
\frac{R_*+R_1}{|B_1|R_*(R_*-R_1)^{d-1}}\Vert f\Vert_{L^1(\partial B_{R_*}(\By_0))} $$
$$(ii)\; \displaystyle\Vert v_e\Vert_{C(\RR^d\setminus
B_{R_2}(\By_0))}\leq \frac{R_2+R_*}{|B_1|R_*(R_2-R_*)^{d-1}}\Vert g\Vert_{L^1(\partial B_{R_*}(\By_0))}$$
where $|B_1|$ denotes the volume of the unit ball $B_1(\By_0)$.
\end{lemma}

\begin{proof}
Without loss of generality, we assume that the three balls
are centered in the origin, i.e., ${\By}_0={\B0}$. In this
condition, form the Poisson formula we have,
\beq
v_i(\Bx)=\frac{1}{|B_1|}\int_{\partial B_{R_*}(\B0)}f(\By)\;\frac{R_*^2-|\Bx|^2}{R_*|\Bx-\By|^d}\;ds_{\By},\mbox{
for } |\Bx|< R_*
\eeq{3-A}
and
\beq
v_e(\Bx)=\frac{1}{|B_1|}\int_{\partial
B_{R_*}(\B0)}g(\By)\;\frac{|\Bx|^2-R_*^2}{R_*|\Bx-\By|^d}\;ds_{\By},\mbox{
for } |\Bx|>R_*
\eeq{4-A}
where $|B_1|$ denotes the volume of the $d$-dimensional unit ball.
Recall that the triangle inequality states,
\beq
|\Bx-\By|\geq||\Bx|-|\By||,\mbox{ for all
}\Bx,\By\in\RR^d
\eeq{5-A}
From (\ref{3-A}) and (\ref{5-A}) we obtain
$$|v_i(\Bx)|\leq \frac{1}{|B_1|}\int_{\partial B_{R_*}(\B0)}|f(\By)|\frac{R_*+|\Bx|}{R_*|R_*-|\Bx||^{d-1}}ds_{\By},\mbox{
for }|\Bx|<R_* $$ Thus
\beq
|v_i(\Bx)|\leq\frac{R_*+R_1}{|B_1|R_*(R_*-R_1)^{d-1}}\int_{\partial
B_{R_*}(\B0)}|f(\By)|ds_{\By},\mbox{ for }|\Bx|\leq R_1
\eeq{6-A}
From (\ref{4-A}) and (\ref{5-A}) similarly we obtain
$$|v_e(\Bx)|\leq \frac{1}{|B_1|}\int_{\partial B_{R_*}(\B0)}|g(\By)|\frac{R_*+|\Bx|}{R_*||\Bx|-R_*|^{d-1}}ds_{\By},\mbox{
for }|\Bx|>R_* $$
and this implies
$$|v_e(\Bx)|\leq\frac{R_*+|\Bx|}{|B_1{R_*||\Bx|-R_*|^{d-1}}}\int_{\partial B_{R_*}(\B0)}|g(\By)|ds_{\By},\mbox{ for }|\Bx|>R_*.$$
By using simple algebra the last inequality becomes
$$|v_e(\Bx)|\leq\frac{1}{|B_1|R_*}\left(\frac{1}{(|\Bx|-R_*)^{d-2}}+\frac{2R_*}{(|\Bx|-R_*)^{d-1}}\right)\int_{\partial
B_{R_*}(\B0)}|g(\By)|ds_{\By},\mbox{ for }|\Bx|>R_*.$$
and, similarly as for (\ref{6-A}) this implies
\beq|v_e(\Bx)|\leq\frac{1}{|B_1|R_*}\left(\frac{1}{(R_2-R_*)^{d-2}}+\frac{2R_*}{(R_2-R_*)^{d-1}}\right)\int_{\partial
B_{R_*}(\B0)}|g(\By)|ds_{\By},\mbox{ for }|\Bx|\geq R_2.\eeq{7-A} The
statement of the Lemma is implied by (\ref{6-A}), and (\ref{7-A}).
\end{proof}

The next result is classical but for the self contained character of the paper we choose to include it here.
\begin{lemma}
    The operator $K$ defined in~\eq{EQ:K} is a compact linear operator from $L^2(\partial B_\Gd(\B0))$ to $\Xi$.
\end{lemma}
\begin{proof}

Let $\{u_n\}$ be a bounded sequence in $L^2(\partial B_\Gd(\B0))$,
i.e.,
\beq
\Vert u_n\Vert_{L^2(\partial B_\Gd(\B0))}\leq C
\eeq{8-A}
with $C$ independent of $n$. Then there exists $u\in L^2(\partial
B_\Gd(\B0))$ such that, up to a subsequence still indexed by $n$, we
have
\beq
    u_n\rightharpoonup u,\mbox{ weekly in } L^2(\partial B_\Gd(\B0))
\eeq{9-A}
Moreover, \eq{8-A} and \eq{9-A} imply,
\beq
\|u\|_{L^2(\partial B_\delta(\B0))}\leq C
\eeq{10-A}
From the continuity of the kernels of $K_1$ and $K_2$ defined at \eq{8} and by using \eq{9-A} we obtain,

\beqa
K_1u_n(\Bx)& \rightarrow & K_1u(\Bx),\mbox{ for all } \Bx\in\partial B_{a'}(\Bx_0)\nonumber\\
& &\nonumber\\
K_2u_n(\Bz)&\rightarrow & K_2u(\Bz),\mbox{ for all } \Bz\in\partial B_{R'}(\B0)
\eeqa{11-A}
Next observe that
\beq
\displaystyle\frac{\partial
\Phi(\Bx,\By)}{\partial\BGv_{\By}}=\frac{1}{|B_1|}\frac{(\Bx-\By)\cdot
\By}{\Gd|\Bx-\By|^d},\mbox{ for }\By\in\partial B_{\Gd}(\B0) \mbox{ and
for any }\Bx\neq\By
\eeq{12-A}
where again $|B_1|$ is the volume of the $d$-dimensional unit ball and we used here that
$\displaystyle\BGv_{\By}=\frac{\By}{|\By|}$ on $\partial B_{\Gd}(\B0)$.
From \eq{8-A}, \eq{9-A}, \eq{10-A} and \eq{12-A} we have
\beqa
|K_1(u_n-u)(\Bx)| & = & \left|\displaystyle \int_{\partial
B_\Gd(\B0)}(u_n-u)(\By)\frac{\partial
\Phi(\Bx,\By)}{\partial\BGv_{\By}}ds_{\By}\right| \nonumber\\
&&\nonumber\\
 &\leq& \Vert u_n-u\Vert_{L^2(\partial
B_\Gd(\B0))}\left\Vert\frac{\partial
\Phi(\Bx,\By)}{\partial\BGv_{\By}}\right\Vert_{L^2(\partial
B_\Gd(\B0))}\nonumber \\
&&\nonumber\\
&\leq& \displaystyle
C\left(\frac{1}{|\Bx_0|-a'-\Gd}\right)^{d-1}\;\;\;\;\mbox{ for
}\Bx\in\partial B_{a'}(\Bx_0)
\eeqa{13-A}
where we used that $$\displaystyle \dist(\partial B_{a'}(\Bx_0), \partial B_\Gd(\B0)) = |\Bx_0|-a'-\Gd>0$$ with $|\Bx_0|-a'-\Gd>0$ following from~\eqref{EQ:Parameters}.

Similarly as for \eq{13-A}, from \eq{8-A}, \eq{9-A} \eq{12-A} we obtain,
\beqa
|K_2(u_n-u)(\Bx)| & = & \left|\displaystyle \int_{\partial
B_\Gd(\B0)}(u_n-u)(\By)\frac{\partial
\Phi(\Bz,\By)}{\partial\BGv_{\By}}ds_{\By}\right|\nonumber\\
&&\nonumber\\
 &\leq& \Vert u_n-u\Vert_{L^2(\partial
B_\Gd(\B0))}\left\Vert\frac{\partial
\Phi(\Bz,\By)}{\partial\BGv_{\By}}\right\Vert_{L^2(\partial
B_\Gd(\B0))}\nonumber\\
&&\nonumber\\
&\leq& \displaystyle C\frac{1}{(R'-\Gd)}\;\;\;\;\mbox{ for
}\Bz\in\partial B_{R'}(\B0)
\eeqa{14-A}
where we used that
$$\displaystyle \dist(\partial B_{R'}(\B0), \partial B_\Gd(\B0)) = R'-\Gd>0$$

From \eq{11-A}, \eq{13-A} and \eq{14-A} and the Lebesgue dominated convergence theorem, we obtain
\beqa
K_1u_n\rightarrow K_1u & &\mbox{ strongly in } L^2(\partial B_{a'}(\Bx_0))\nonumber\\
&&\nonumber\\
K_2u_n\rightarrow K_2u & &\mbox{ strongly in } L^2(\partial B_{R'}(\B0))
\eeqa{15-A}
Then from the definition of the operator $K$ and from \eq{15-A} we conclude that
$$Ku_n\rightarrow Ku \mbox{ strongly in }\Xi$$ and this implies the statement of the lemma.
\end{proof}

%%%%%%%%%%%%%%%%%%%%%%%%%%%%%%%%%%%%%%%%%%%%%%%%%%%%%%%%%%%%%%%%%%
%%%%%%%%%%%%%%%%%%%%%%%%%%%%%%%%%%%%%%%%%%%%%%%%%%%%%%%%%%%%%%%%%%
%\bibliography{/home/ren/Academic/Bibliography/BIB-KR,/home/ren/Academic/Bibliography/BIB-YH}

\begin{thebibliography}{10}

\bibitem{Alu}
{\sc A.~Alu and N.~Engheta}, {\em Plasmonic and metamaterial cloaking: physical
  mechanism and potentials}, J. Opt. A: Pure Appl. Opt, 10, (2008).


\bibitem{B}
{\sc H.~Brezis}, {\em Analyse functionnelle. Theorie et applications}, Dunod,
  Paris, France, 2nd~ed., (1999).

\bibitem{Chan3}
{\sc H.~Chen and C.~Chan}, {\em Acoustic cloaking and transformation
  acoustics}, J. Phys. D: Appl. Phys., 43, (2010).


\bibitem{Chan2}
{\sc H.~Chen, C.~T. Chan, and P.~Sheng}, {\em Transformation optics and
  metamaterials}, Nature Materials, 9, pp.~387--396, (2010).

\bibitem{CoKr-Book98}
{\sc D.~Colton and R.~Kress}, {\em Inverse Acoustic and Electromagnetic
  Scattering Theory}, Springer-Verlag, New York, (1998).

\bibitem{Cummer}
{\sc S.~A. Cummer, B.-I. Popa, D.~Schurig, D.~R. Smith, J.~Pendry, M.~Rahm, and
  A.~Starr}, {\em Scattering theory derivation of a 3d acoustic cloaking
  shell}, Phys. Rev. Lett., 100, (2008).

\bibitem{Green2}
{\sc A.~Greenleaf, Y.~Kurylev, M.~Lassas, and G.~Uhlmann}, {\em Invisibility
  and inverse problems}, Bull. Amer. Math. Soc., 46, pp.~55--97, (2009).

\bibitem{Green1}
{\sc A.~Greenleaf, M.~Lassas, and G.~Uhlmann}, {\em Anisotropic conductivities
  that cannot be detected by eit}, Physiol. Meas., 24, pp.~413--419, (2003).

\bibitem{OMV4}
{\sc F.~ Guevara~Vasquez, G.~W. Milton, D. Onofrei, P. Seppecher},
 {\em Transformation
  elastodynamics and active exterior acoustic cloaking}, Acoustic
  metamaterials: Negative refraction, imaging, lensing and cloaking, arXiv:1105.1221, (2011).

\bibitem{OMV1}
{\sc F.~Guevara~Vasquez, G.~W. Milton, and D.~Onofrei}, {\em Active exterior
  cloaking}, Phys. Rev. Lett., 103, (2009).
%\newblock 073901.

\bibitem{OMV2}
\leavevmode\vrule height 2pt depth -1.6pt width 23pt, {\em
Broadband exterior cloaking}, Optics Express, 17,
pp.~14800--14805, (2009).

\bibitem{OMV3}
\leavevmode\vrule height 2pt depth -1.6pt width 23pt, {\em
Exterior cloaking
  with active sources in two dimensional acoustics}, accepted Wave Motion, arXiv:1009.2038, (2011).

\bibitem{Kirsch-Book96}
{\sc A.~Kirsch}, {\em An Introduction to the Mathematical Theory of Inverse
  Problems}, Springer-Verlag, New York, (1996).

\bibitem{Kohn}
{\sc R.~Kohn, D.~Onofrei, M.~Vogelius, and M.~Weinstein}, {\em Cloaking via
  change of variables for the helmholtz equation at fixed frequency}, Comm.
  Pure. Appl. Math., 63, pp.~973--1016, (2000).

\bibitem{Kress-Book99}
{\sc R.~Kress}, {\em Linear Integral Equations}, Applied Mathematical Sciences,
  Springer-Verlag, New York, 2nd~ed., (1999).

\bibitem{Chan}
{\sc Y.~Lai, H.~Chen, Z.-Q. Zhang, and C.~T. Chan}, {\em Complementary media
  invisibility cloak that cloaks objects at a distance outside the cloaking
  shell}, Phys. Rev. Lett., 102, (2009).

\bibitem{Ulf2}
{\sc U.~Leonhardt}, {\em Notes on conformal invisibility devices}, New J.
  Phys., 8, p.~118, (2006).

\bibitem{Ulf1}
\leavevmode\vrule height 2pt depth -1.6pt width 23pt, {\em Optical conformal
  mapping}, Science, 312, pp.~1777--1780, (2006).

\bibitem{Miller}
{\sc D.~A.~B. Miller}, {\em On perfect cloaking}, Opt. Express,
14, pp.~12457--12466, (2006).

\bibitem{Mil1}
{\sc G.~W. Milton and N.-A.~P. Nicorovici}, {\em On the cloaking effects
  associated with anomalous localized resonance}, Proc. R. Soc. Lon. Ser. A.
  Math. Phys. Sci., 462, pp.~3027--3059, (2006).

\bibitem{Mil3}
{\sc G.~W. Milton, N.-A.~P. Nicorovici, R.~C. McPhedran, K.~Cherednichenko, and
  Z.~Jacob}, {\em Solutions in folded geometries, and associated cloaking due
  to anomalous resonance}, New J. Phys., 10, (2008).


\bibitem{Mil2}
{\sc N.-A.~P. Nicorovici, G.~Milton, R.~C. McPhedran, and L.~C. Botten}, {\em
  Quasistatic cloaking of two-dimensional polarizable discrete systems by
  anomalous resonance}, Opt. Express, 15, pp.~6314--6323, (2007).

\bibitem{Pendry}
{\sc J.~B. Pendry, D.~Schurig, and D.~R. Smith}, {\em Controlling
  electromagnetic fields}, Science, 312, pp.~1780--1782, (2006).

\bibitem{Shalaev}
{\sc I.~I. Smolyaninov, V.~N. Smolyaninova, A.~V. Kildishev, and V.~M.
  Shalaev}, {\em Anisotropic metamaterials emulated by tapered waveguides:
  Application to optical cloaking}, Phys. Rev. Lett., 103, (2009).

\bibitem{Hoai} Hoai-Minh Nguyen,\textit{ Approximate cloaking for the Helmholtz equation via transformation optics and consequences for perfect cloaking},
to appear in Comm. Pure Appl. Math , Comm. Pure. Appl. Math., Volume 65, Issue 2, pages 155–186, (2012).

\bibitem{Hongyu3} H. Y. Liu, \textit{Virtual reshaping and invisibility in obstacle scattering}, Inverse
Problems, 25, 045006, (2009).

\bibitem{Hongyu1}Hongyu Liu, Hongpeng Sun, \textit{Enhanced Near-cloak by FSH Lining}, online at arXiv:1110.0752, (2011).

\bibitem{Hongyu2}Hongyu Liu, Ting Zhou, \textit{On Approximate Electromagnetic Cloaking by Transformation Media }, SIAM J. Appl. Math., 71, pp. 218--241, (2011).

\bibitem{Dol}{\sc L.~S. Dolin}, \textit{On a possibility of comparing three-dimensional electromagnetic systems with
inhomogeneous filling}, Izv. Vyssh. Uchebn. Zaved. Radiofiz. 4, 964, (1961).

\bibitem{Post} E. J. Post, \textit{Formal structure of electromagnetics}, North-Holland (1962).

\bibitem{LaxN} M. Lax and D. F. Nelson, “Maxwell equations in material form,” Phys. Rev. B 13, 1777, (1976).

\bibitem{Kang1} H. Ammari, Hyeonbae Kang, Hyundae Lee, Mikyoung Lim,\textit{Enhancement of near-cloaking. Part II: the Helmholtz equation}, To appear in Communications in Mathematical Physics, (2012).

\bibitem{Kang2} H. Ammari, J. Garnier, V. Jugnon, H. Kang, H. Lee, and M. Lim, \textit{Enhancement of near-cloaking. Part III: numerical simulations, statistical stability, and related questions}. To appear in Contemporary Mathematics (2012).

\bibitem{Kang3} H. Ammari, H. Kang, H. Lee, and M. Lim, \textit{Enhancement of near-cloaking using generalized polarization tensors vanishing structures. Part I: The conductivity problem.} to appear in Communications in Mathematical Physics (2012).

\bibitem{Valentine}J. Valentine, Z. Shuang, T. Zentgraf, Z. Xiang  \textit{ Development of Bulk Optical Negative Index Fishnet Metamaterials: Achieving a Low-Loss and Broadband Response Through Coupling}, Proceedings of the IEEE, Vol. 99,   Iss. 10
pp. 1682 - 1690, (2011).

\bibitem{Gunther1} Allan Greenleaf, Yaroslav Kurylev, Matti Lassas, Gunther Uhlmann, \textit{ Schrodinger's Hat: Electromagnetic, acoustic and quantum amplifiers via transformation optics}, online at arxiv.org: arXiv:1107.4685v1, 2011.

\bibitem{Ono9} D. Onofrei,\textit{ On the stability of the active field manipulation design}, in preparation, (2012).

	\bibitem{Peake} N. Peake, D.G. Crighton, \textit{Active control of sound}, Annu. Rev. Fluid Mech., vol.32, 137-164, (2000).
 	  	
 	  	\bibitem{Elliot} S.J. Elliot, P.A. Nelson, \textit{The active control of sound}, Electronics and Comm. Engineering Journal, August, (1990).
 	  	
 	  	\bibitem{Tsynkov} J. Loncaric, V.S. Ryaben'kii, S.V. Tsynkov, \textit{Active shielding and control of environmental noise}, technical report, NASA/CR-2000-209862, ICASE Report No. 2000-9, (2000).
 	  	
 	  	\bibitem{T1} J. Loncaric, S.V. Tsynkov,\textit{ Quadratic optimization in the problems of active control of sound}, Appl. NUm. Math., vol. 52, 381-400, 2005.
 	  	
 	  	\bibitem{Peterson} A.W. Peterson, S.V. Tsynkov, \textit{Active control of sound for composite regions}, SIAM J. Appl. Math, vol. 67, Iss. 6, pp 1582-1609, (2007).
 	  	
 	  	\bibitem{Fuller} C.R. Fuller, A.H. von Flotow, \textit{ Active control of sound and vibration}, IEEE, (1995).
 	  	
 	  	\bibitem{Leug} P. Leug, \textit{Process of silencing sound oscillations}, U.S. patent no. 2043416, (1936).
 	  	
 	  	\bibitem{Olson-May} H.F. Olson, E.G. May, \textit{Electronic sound absorber}, J. Acad. Soc. America, vol. 25, pp. 1130-1136, (1953).

         \bibitem{Gunther2} A. Greenleaf, Y. Kurylev, M. Lassas, G. Uhlmann, \textit{Cloaking a Sensor via Transformation Optics}, Physical Review E 83, 016603 (2011).
         
         \bibitem{ALU4} Giuseppe Castaldi, Ilaria Gallina, Vincenzo Galdi, Andrea Alu, Nader Engheta, \textit{Power scattering and absorption mediated by cloak/anti-cloak interactions: A transformation-optics route towards invisible sensors}, JOSA B, Vol. 27, Iss. 10, pp. 2132-2140, (2010).
         
         \bibitem{CTchan-num} H. H. Zheng, J. J. Xiao, Y. Lai, C. T. Chan \textit{ Exterior optical cloaking and illusions by using active sources: A boundary element perspective}, Phys. Rev. B, Vol. 81, Issue 19, 2010.
         
\end{thebibliography}
%\bibliographystyle{siam}

%%%%%%%%%%%%%%%%%%%%%%%%%%%%%%%%%%%%%%%%%%%%%%%%%%%%%%%%%%%%%%%%%%
%%%%%%%%%%%%%%%%%%%%%%%%%%%%%%%%%%%%%%%%%%%%%%%%%%%%%%%%%%%%%%%%%%

\end{document}